\documentclass[11pt,a4paper,reqno]{article}
\usepackage[latin1]{inputenc}  
\usepackage[T1]{fontenc}
\usepackage[english]{babel}
\usepackage{epsfig,epsf,graphicx}
\usepackage{amsfonts,amssymb,latexsym,a4wide,xspace}
\usepackage{amsmath,delarray,enumerate,calc,amsthm}
\usepackage{bbm}
\usepackage{txfonts}
\usepackage{paralist}
\usepackage{authblk}
\usepackage{tikz-cd}
\usepackage{stmaryrd}
\usepackage[mathscr]{euscript}
\newcommand{\oplam}{\mbox{\Large $\curlywedge$}}

\newcommand{\id}{\operatorname{id}}
\newcommand{\supp}{\operatorname{supp}}

\newcommand{\cB}{{\mathcal B}}
\newcommand{\cF}{{\mathcal F}}
\newcommand{\cG}{{\mathcal G}}

\newcommand{\cM}{{\mathcal M}}
\newcommand{\cMG}{\cM^{\scriptscriptstyle G}}

\newcommand{\LL}{{\mathscr L}}
\newcommand{\BB}{{\mathscr B}}
\newcommand{\R}{{\mathbbm R}}
\newcommand{\N}{{\mathbbm N_0}}

\newcommand{\Z}{{\mathbbm Z}}
\newcommand{\A}{{\mathbbm A}}

\newcommand{\T}{{\mathbbm T}}
\newcommand{\X}{{\mathbbm X}}
\newcommand{\1}{{\mathbbm 1}} 

\newcommand{\inn}{\operatorname{int}}
\newcommand{\card}{\operatorname{card}}

\newcommand{\hx}{{\hat{x}}}

\newcommand{\hX}{{\hat{X}}}

\newcommand{\hXmax}{\hX_{max}}
\newcommand{\hXmaxG}{\hX_{max}^G}
\newcommand{\hT}{\hat{T}}
\newcommand{\Y}{{\mathcal Y}}
\newcommand{\muW}{\mu_{\scriptscriptstyle W}}
\newcommand{\nuW}{\nu_{\scriptscriptstyle W}}

\newcommand{\nuWzero}{\nu_{\scriptscriptstyle W_0}}
\newcommand{\nuWbdzero}{\nu_{\scriptscriptstyle \partial W_0}}
\newcommand{\nuWG}{\nu_{\scriptscriptstyle W}^{\scriptscriptstyle G}}
\newcommand{\nuWGprime}{\nu_{\scriptscriptstyle W'}^{\scriptscriptstyle G}}
\newcommand{\nuWGzero}{\nu_{\scriptscriptstyle W_0}^{\scriptscriptstyle G}}
\newcommand{\GnuW}{{\cG\nuW}}
\newcommand{\GnuWG}{{\cG\nuWG}}
\newcommand{\MW}{\cM_{\scriptscriptstyle W}}
\newcommand{\MWprime}{\cM_{\scriptscriptstyle W'}}
\newcommand{\MWzero}{\cM_{\scriptscriptstyle W_0}}
\newcommand{\GMW}{\cG\cM_{\scriptscriptstyle W}}
\newcommand{\MWG}{\cM^{\scriptscriptstyle G}_{\scriptscriptstyle W}}
\newcommand{\MWGprime}{\cM^{\scriptscriptstyle G}_{\scriptscriptstyle W'}}
\newcommand{\MWGzero}{\cM^{\scriptscriptstyle G}_{\scriptscriptstyle W_0}}
\newcommand{\GMWG}{\cG\cM^{\scriptscriptstyle G}_{\scriptscriptstyle W}}
\newcommand{\piG}{\pi^{\scriptscriptstyle G}}
\newcommand{\piH}{\pi^{\scriptscriptstyle H}}
\newcommand{\piGH}{\pi^{\scriptscriptstyle G\times H}}
\newcommand{\pihX}{\pi^{\scriptscriptstyle \hX}}
\newcommand{\pihXG}{\pi^{\scriptscriptstyle \hX\times G}}
\newcommand{\CW}{C_{\scriptscriptstyle W}}
\newcommand{\CWprime}{C_{\scriptscriptstyle W'}}
\newcommand{\CWzero}{C_{\scriptscriptstyle W_0}}
\newcommand{\CWbdzero}{C_{\scriptscriptstyle \partial W_0}}
\newcommand{\ZW}{Z_{\scriptscriptstyle W}}
\newcommand{\ZWbdzero}{Z_{\scriptscriptstyle \partial W_0}}
\newcommand{\0}{\underline{0}}
\newcommand{\QM}{Q_{\cM}}

\newcommand{\QGM}{Q_{\cG\cM}}
\newcommand{\QMG}{Q_{\cM^G}}
\newcommand{\QGMG}{Q_{\cG\cM^G}}
\newcommand{\oneone}{1-1\xspace}
\newcommand{\dL}{\mathrm{dens}(\LL)}

\newtheorem {definition}{Definition}[section] 
\newtheorem {lemma}[definition]{Lemma}
\newtheorem{theorem}{Theorem}
\newtheorem {bemerkung}[definition]{Remark}
\newtheorem{proposition}[definition]{Proposition}
\newtheorem {corollary}{Corollary}
\newtheorem{beispiel}[definition]{Example}
\newtheorem{frage}[definition]{Problem}
\newenvironment{remark} {\begin{bemerkung} \normalfont }{\end{bemerkung}}
\newenvironment{example} {\begin{beispiel} \normalfont }{\end{beispiel}}

\begin{document}

\title{Dynamics on the graph of the torus parametrisation
}
\author{Gerhard Keller and Christoph Richard\;
\thanks{These notes profited enormously from talks by and/or discussions with Michael Baake, Tobias Hartnick, Johannes Kellendonk,  Mariusz Lema\'n{}czyk, Daniel Lenz, Tobias Oertel-J\"ager and Nicolae Strungaru, from several inspiring workshops in the framework of the DFG Scientific Network ``Skew Product Dynamics and Multifractal Analysis'' organised by Tobias Oertel-J\"ager, and finally from very helpful and supporting comments of an anonymous referee.}}
\affil{Department Mathematik, Universit\"at Erlangen-N\"urnberg}
\date{\today}


\maketitle


\begin{abstract}
Model sets are projections of certain lattice subsets. It was realised by Moody that dynamical properties of such a set are induced from the torus associated with the lattice. We follow and extend this approach by studying dynamics on the graph of the map which associates lattice subsets to points of the torus and then transferring the results to their projections.
This not only leads to transparent proofs of known results on model sets, but we also obtain new results on so-called weak model sets. In particular we  prove pure point dynamical spectrum for the hull of a weak model set of maximal density together with the push forward of the torus Haar measure under the torus parametrisation map, and we derive a formula for its pattern frequencies.
\end{abstract}

\renewcommand{\thefootnote}{\fnsymbol{footnote}} 
\renewcommand{\thefootnote}{\arabic{footnote}} 

\section{Introduction}

Let $G$ and $H$ be locally compact second countable abelian groups like, e.g., $\Z^d$ or $\R^d$. Each pair $(\LL,W)$, consisting of a lattice $\LL\subset G\times H$ and a relatively compact subset $W$ of $H$, also called the window, defines a weak model set $\Lambda(\LL,W)$ as the set of all points $x_G\in G$ for which there exists a point $x_H\in W$ such that $(x_G,x_H)\in\LL$. Under additional assumptions on $\LL$ and $W$, the structure of such sets has attracted much attention in the mathematics and physics literature.

Model sets, which satisfy $\inn(W)\ne\emptyset$, have been introduced by Meyer \cite{Meyer70,Meyer72} in the context of number theory and harmonic analysis. After the discovery of quasicrystal alloys, model sets as a mathematical abstraction of these structures have been advertised and developed by Moody, see e.g.~\cite{Moody97,Moody00}. We also mention the recent  comprehensive
monograph by Baake and Grimm \cite{BaakeGrimm13}. The name \textit{weak model set} was coined by Moody \cite{Moody2002}. In fact weak model sets have been initially studied by Schreiber \cite{Schreiber71,Schreiber73}, see also \cite{HuckRichard15} for background.

For a given cut-and-project scheme $(G,H,\LL)$ with window $W$, the torus $\hX$ associated with the lattice $\LL$ may be used to parametrise model sets  arising from translations in $G$ and shifts of the window. This is named the \textit{torus parametrisation} in \cite{BHP97, HRB97}, which was also
studied in \cite{Robinson1996, Schlottmann00, FHK2002, Robinson2007, BLM07, BHP16}. For weighted model sets with continuous compactly supported weight functions, it was investigated in \cite{LenzRichard07}. Some results about the torus parametrisation of general weak model sets appear in \cite{Moody2002}.

In this paper we consider weak model sets with compact windows, along with their associated dynamical systems, also called their hulls. The hull is the vague translation orbit closure of the weak model set, if one identifies the point set with a measure which puts unit mass to every point. In contrast to previous approaches, we will study hulls in the space $\cM$ of locally finite measures on $G\times H$ instead of $G$ and in an even larger space, namely on the graph of the torus parametrisation, which takes values in $\hX\times \cM$. This highlights connections with the dynamics of skew-product dynamical systems with monotone fibre maps, because the latter system is of this type. Our approach
simplifies many arguments and makes the origin of certain standard assumptions on model sets more transparent, which only enter when projecting from $G\times H$ to $G$. 
Extensions to relatively compact windows are discussed as well.

A central result in the theory of model sets with topologically regular windows is almost automorphy of their hull, if the window boundaries have empty intersection with the projected lattice, see Robinson \cite{Robinson2007, Robinson1996, Robinson2004} and also \cite{FHK2002, ABKL14}. We give an alternative proof of this result by our abstract dynamical systems approach. We will however not investigate in this paper how almost automorphy characterises certain classes of model sets, compare \cite[Thm.~3.16]{ABKL14}.
On the measure theoretic side,  a central result by Schlottmann~\cite{Schlottmann00}
relates to so-called regular model sets. These are weak model sets whose window boundaries have vanishing Haar measure. Schlottmann's result expresses that the hull of a regular model set has pure point dynamical spectrum. This, in turn, implies that any regular model set has pure point diffraction spectrum \cite{LMS02, BaakeLenz2004}. A generalisation for weak model sets was obtained by Moody \cite{Moody2002}. His result, when restricted to weak model sets with compact window, states that ``almost all'' weak model sets have pure point diffraction spectrum and so-called maximal density. The importance of weak model sets of maximal density was realised recently by Strungaru \cite{S14}, who argued that such model sets have pure point diffraction spectrum. Within our setup, we can prove that the hull of  any weak model set of maximal density, equipped with the push forward of the torus Haar measure by the torus parametrisation, has pure point dynamical spectrum. Our arguments follow from structural assumptions on the torus parametrisation and from a careful revision of Moody's arguments. They crucially rely on dynamical properties of weak model sets of maximal density. This will be compared to the recent work by Baake, Huck and Strungaru \cite{BaakeHuckStrungaru15}, who obtain similar dynamical results for weak model sets with a non-empty relatively compact window, which are of maximal or minimal density.  Their approach is approximation with regular model sets. 

There are many examples of point sets in $G$ that have a - sometimes hidden - description as weak model set, see e.g.~the monograph \cite{BaakeGrimm13}.
We do not consider here the important question of reconstructing the internal space, the lattice and the window from the point set, but refer to the results Meyer \cite[Ch.~II.5]{Meyer72}, Schreiber \cite[Thm.~2]{Schreiber73}, Schlottmann \cite[Thm.~6]{Schlottmann1998} and Aujogue~\cite[Thm.~3.16]{ABKL14}. Compare the approach by Baake and Moody \cite{BaakeMoody2004}, which is based on the autocorrelation of the point set, and which has recently been revisited by Strungaru \cite{Str15}.

In fact certain such weak model sets of maximal density have recently attracted attention due to their intimate connection with the M\"obius function from number theory. These are the $k$-free lattice points and, more generally, the $\BB$-free systems, see \cite{Sarnak2011, BaakeHuck14,ALR2013} and references therein. For the visible lattice points, pure point diffraction spectrum was shown already in \cite{BMP00}. Pure point dynamical spectrum for the hull of the square-free integers was shown only recently in \cite{CS13} without referring to weak model sets, and in \cite{BaakeHuck14} using weak model sets. Whereas these results were obtained by explicit calculation, 
pure point dynamical spectrum follows from structural arguments within our approach. This indicates that one may take advantage of the underlying weak model set structure in order to further analyse the simplex of invariant probability measures for these systems.

The main contribution of this paper is the systematic use of the dynamical system that arises from the graph of the torus parametrisation $\nuW:\hX\to\cM$, where $\nuW(\hx)$ is the configuration on $G\times H$ defined by the window $W$ and the lattice $\LL$ shifted by $\hx=x+\LL$. It has support $\supp(\nuW(\hx))=(x+\LL)\cap(G\times W)$, as illustrated in Figure~\ref{figure:cpscheme}. This point of 
view separates most dynamical considerations from purely model dependent technical problems that are unavoidable when passing from configurations in $G\times H$ to their projections in $G$, see Figure~\ref{figure:cpscheme}. \footnote{In the existing literature the term torus parametrisation is mostly used for the map $\nuWG$ illustrated in  Figure~\ref{figure:cpscheme} on p.~\pageref{figure:cpscheme} (or for its inverse, when it exists), so that the purely dynamical aspects are always intertwined with problems solely due to the passage from $G\times H$ to $G$.}

\begin{figure}[tb]
\begin{center}
\begin{minipage}[b]{\textwidth}
\center{\hspace{6ex}\epsfig{file=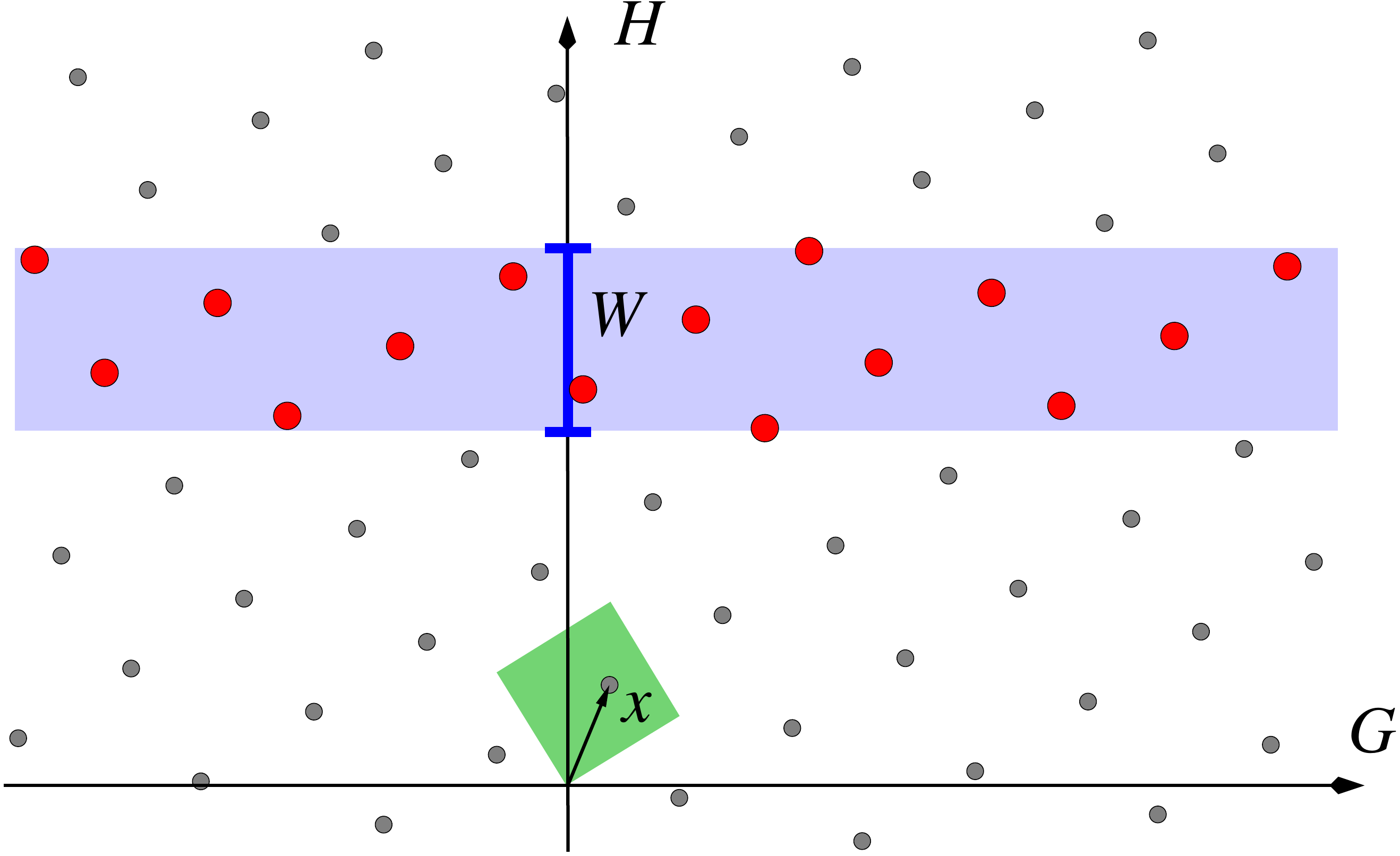,width=6cm,height=4cm}
\hfill\epsfig{file=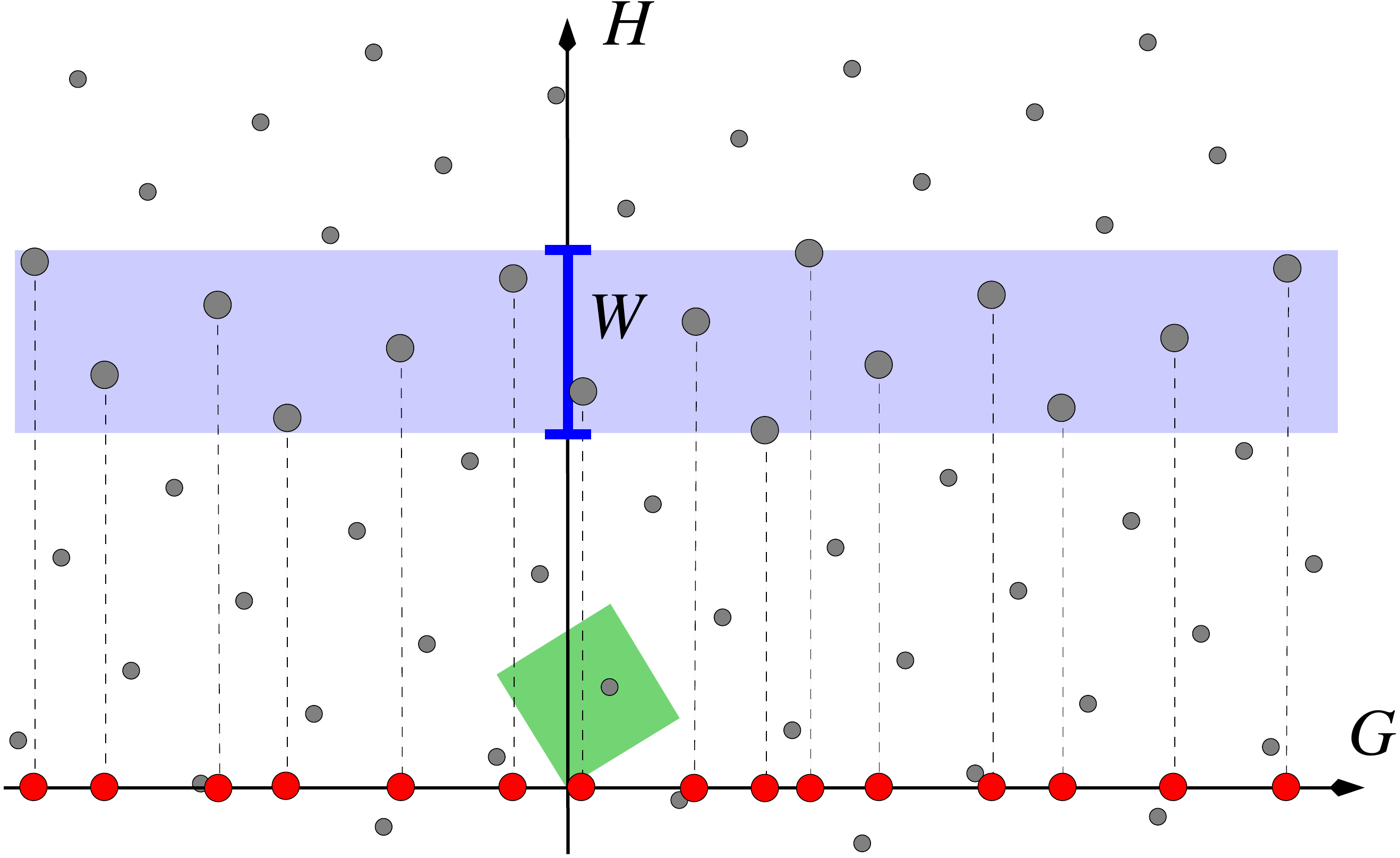,width=6cm,height=4cm}\hspace{6ex}}\\
\hspace*{0.6cm}
Red/dark dots: support of $\nuW(\hx)$
\hspace*{3.4cm}
Red/dark dots: support of $\nuWG(\hx)$
\end{minipage}
\end{center}
\caption{Shifted lattice $\LL$ (black dots), fundamental domain of $\LL$ (green/gray), and the supports of $\nuW(\hx)$ and $\nuWG(\hx)$.\label{figure:cpscheme}}
\end{figure}

On the general dynamical level (that is before passing to configurations on $G$)
\begin{itemize}
\item we obtain as a warm-up some relatives to well known topological results
on maximal equicontinuous factors and the lack of weak mixing (Theorem~\ref{theo:projection-top}),
\item we can introduce a general notion of Mirsky measures, namely the push forward of Haar measure on $\hX$ under the map $\nuW$, compare \cite{Abdalaoui2014,CS13,Kulaga-Przymus2014,Sarnak2011},
\item we show that the systems equipped with this Mirsky measure have  pure point dynamical spectrum (Theorem~\ref{theo:projection-mth}),
\item we prove strict ergodicity when $m_H(\partial W)=0$ (Theorem~\ref{theo:projection-mth}),
\item we identify the Mirsky measure as the unique invariant measure 
with maximal density for typical configurations (Theorem~\ref{theo:Moody-extension}),
\item we show that the configurations with maximal density are precisely the generic points for the Mirsky measure (Theorem~\ref{theo:Moody-topogical}),
\item and we deduce from this a formula for the pattern frequencies of configurations with maximal density (Remark~\ref{remark:patterns}), which is also discussed in \cite[Rem.~5]{BaakeHuckStrungaru15}.
\end{itemize}
While the measure theoretic assertions from this list ``survive'' the passage to configurations on $G$ even if the relevant projection is not \oneone, some finer information can be transferred in this way only if that projection is \oneone when restricted to sufficiently large subsets. This issue is discussed in Section~\ref{sec:examples} for
model sets with interval windows, topologically regular windows
and $\BB$-free systems.

In the next section we define the basic objects of this paper and describe the relations between the various dynamical systems entering the scene. Section~\ref{sec:main-results}
contains the main theorems of this paper as well as a number of auxiliary results that elucidate our approach to study the dynamics on the graph of $\nuW$.
The proofs of most results are deferred to Sections~\ref{sec:basics} - \ref{sec:Moody-proof}. We finish with an outlook to further perspectives of our approach in Section~\ref{sec:outlook}.

\section{The setting}
\subsection{Assumptions and notations}\label{assnot}
Certain spaces and mappings are needed for the construction of weak model sets. The following assumptions will be in force in any of our statements below.
\begin{enumerate}[(1)]
\item $G$ and $H$ are \emph{locally compact second countable abelian groups} with Haar measures $m_G$ and $m_H$. Then the product group $G\times H$ is locally compact second countable abelian as well, and we choose $m_{G\times H}=m_G\times m_H$ as Haar measure on $G\times H$. 
\item $\LL\subseteq G\times H$ is a \emph{cocompact lattice}, i.e., a discrete subgroup whose quotient space $(G\times H)/\LL$ is compact. Thus $(G\times H)/\LL$ is a compact second countable abelian group.
Denote by $\piG:G\times H\to G$ and $\piH:G\times H\to H$ 
the canonical projections. We assume that 
$\piG|_\LL$ is \oneone and that
$\piH(\LL)$ is dense in $H$.\footnote{Denseness of $\piH(\LL)$ can be assumed without loss of generality by passing from $H$ to the closure of $\piH(\LL)$. In that case $m_H$ must be replaced by $m_{\,\overline{\piH(\LL)}}$.} 
\item $G$ acts on $G\times H$ by translation: $T_gx=(g,0)+x$.
\item Let $\hX:=(G\times H)/\LL$.  As we assumed that $\hX$ is compact, there is a measurable relatively compact fundamental domain $X\subseteq G\times H$ such that $x\mapsto x+\LL$ is a bijection between $X$ and $\hX$. Elements of $G\times H$ (and hence also of $X$) are denoted as $x=(x_G,x_H)$, elements of $\hX$ as 
$\hx$ or as $x+\LL=(x_G,x_H)+\LL$, when a representative $x$ of $\hx$ is to be stressed. We normalise the Haar measure $m_\hX$ on $\hX$ such that $m_\hX(\hX)=1$. Thus $m_\hX$ is a probability measure. 

\item The \emph{window} $W$ is a compact subset of $H$. This is a more restrictive assumption than the one made originally for weak model sets in \cite{Moody2002}, where only relative compactness and measurability of $W$ are assumed. 
We choose to assume compactness of $W$, because this guarantees strong structural results and a more coherent presentation. Results for relatively compact windows are however discussed in Remark~\ref{rem:rc}. (Recall that (full) model sets according to \cite{Moody2002} are weak model sets where $\inn(W)\ne\emptyset$.)
\end{enumerate}

\subsection{Consequences of the assumptions}\label{en:ass}
We list a few facts from topology and measure theory that follow from the above assumptions.  We will call any neighborhood of the neutral element in an abelian topological group a \textit{zero neighborhood}.

\begin{enumerate}[(1)]
\item Being locally compact second countable abelian groups, $G$, $H$ and $G\times H$  are  metrisable with a translation invariant metric with respect to which they are complete metric spaces. In particular they have the Baire property.
As such groups are $\sigma$-compact, $m_G$, $m_H$ and $m_{G\times H}$ are $\sigma$-finite.
\item As $G\times H$ is $\sigma$-compact, the lattice $\LL\subseteq G\times H$ is at most countable.  Note that $G\times H$ can be partitioned by shifted copies of the relatively compact fundamental domain $X$. This means that $\LL$ has a positive finite point density $\dL=1/m_{G\times H}(X)$. We thus have $m_\hX(\hat A)=\dL\cdot m_{G\times H}(X\cap (\pihX)^{-1}(\hat A))$ for any measurable $\hat A\subseteq \hX$, where $\pihX: G\times H\to \hX$ denotes the quotient map. As a factor map between topological groups, $\pihX$ is open.
\item As $\LL$ is a discrete group, there is an open
zero neighbourhood $V\subseteq G\times H$ whose closure is compact and for which all sets $\overline V+x$ $(x\in \LL)$ are pairwise disjoint.
\item $\LL$ acts on $(H,m_H)$ by $h\mapsto\ell_H+ h$ metrically transitively, i.e., for every measurable $A\subseteq H$ such that $m_H(A)>0$ there exist at most countably many $\ell^i\in \LL$ such that $m_H((\bigcup_i (\ell_H^i+A))^c)=0$,  see \cite[Chapter 16, Example 1]{Kharazishvili2009}.
\item The action $\hT_g:\hx\mapsto(g,0)+\hx$ of $G$ on $\hX$ is minimal. Indeed:
let $x+\LL,y+\LL\in\hX$ be arbitrary. Choose a sequence $(\ell_n)_n$ from $\LL$ such that $\ell_{n,H}\to y_H-x_H$ and let $g_n=y_G-\ell_{n,G}-x_G$. Then
\begin{displaymath}
\hT_{g_n}(x+\LL)
=
(g_n,0)+x+\LL
=
(g_n,0)+x+\ell_n+\LL
=
(y_G,\ell_{n,H}+x_H)+\LL
\to
y+\LL\ .
\end{displaymath}
This shows that the $G$-orbit of $x+\LL$ is dense in $\hX$, i.e., minimality of the action $\hT$. This implies that $\hX$ with its natural action is uniquely ergodic, see e.g.~\cite[Prop.~1]{Moody2002}.

\item Denote by $\cM$ and $\cMG$ the spaces of all locally finite measures on $G\times H$ and $G$, respectively. They are endowed with the topology of vague convergence. 
As $G$ and $G\times H$ are complete metric spaces, this is a Polish topology, see \cite[Theorem A.2.3]{Kallenberg2001}. 
\end{enumerate}

\subsection{The objects of interest}
The pair $(\LL,W)$ assigns to each point $\hx\in \hX$ a discrete point set in $G\times H$. 
We will identify such point sets $P$ with the measure $\sum_{y\in P}\delta_y\,\in\cM$
and call these objects \emph{configurations}. More precisely:
\begin{enumerate}[(1)]
\item For $\hx=x+\LL\in\hX$ define the configuration
\begin{equation}\label{eq:nuW-def}
\nuW(\hx):=\sum_{y\in (x+\LL)\cap(G\times W)}\delta_y\ .
\end{equation}
It is important to understand $\nuW$ as a map from $\hX$ to $\cM$. 
The canonical projection $\piG:G\times H\to G$ projects measures $\nu\in\cM$ to measures $\piG_*\nu$ on $G$ defined by $\piG_*\nu(A):=\nu((\piG)^{-1}(A))$. We abbreviate 
\begin{equation}\label{eq:nuWG-def}
\nuWG:=\piG_*\circ\nuW:\hX\to\cMG
\end{equation}
The configurations $\nuWG(\hx)$ on $G$ are sometimes called \emph{Dirac combs} in the literature. See Figure~\ref{figure:cpscheme} for a visualisation of these two maps.
The set $\nuWG(\hX)$ is called ``the model set family associated to the window $W$'' and denoted by ${\mathscr M}(W)$ in \cite[Eqn.~(4--2)]{Robinson2007}. 
\item Denote the graph of $\nuW:\hX\to\cM$ by $\GnuW=\{(\hx,\nuW(\hx)): \hx\in\hX\}$
and that of $\nuWG:\hX\to\cMG$ by $\GnuWG=\{(\hx,\nuWG(\hx)): \hx\in\hX\}$.
\item\label{item:spaces}
 Denote by 
\begin{itemize}
\item[-] $\MW$ the vague closure of $\nuW(\hX)$ in $\cM$,
\item[-] $\GMW$ the vague closure of $\GnuW$ in $\hX\times \cM$,
\item[-] $\MWG$ the vague closure of $\nuWG(\hX)$ in $\cMG$,
\item[-] $\GMWG$ the vague closure of $\GnuWG$ in $\hX\times \cMG$.
\end{itemize}
The group $G$ acts continuously by translations on all these spaces:
\begin{displaymath}
(S_g\nu)(A):=\nu(T_g^{-1}A)=\nu(T_{-g}A), \quad S_g(\hx,\nu):=(\hT_g\hx,S_g\nu)\ .
\end{displaymath}
Here we used the same notation  $S_g$ for translations on $\MW$, $\MWG$, $\GMW$ and $\GMWG$, as the meaning will always be clear from the context. 
\item As $\nuW(\hx)(T_{-g}A)=(S_g\nuW(\hx))(A)=\nuW(\hT_g\hx)(A)$ by Lemma~\ref{lemma:basic-0}, it is obvious that all $\nuW(\hx)$ are uniformly translation bounded, and it follows from \cite[Theorem 2]{BaakeLenz2004} that all four spaces from item (\ref{item:spaces}) are compact.
\end{enumerate}

\subsection{Spaces and factor maps}
On the group level we will work with the following sets and canonical projections:
\begin{equation*}
\begin{tikzcd} 
\ &\hX\times G\times H \arrow[swap]{ld}{\pihX}\arrow{dd}{\pihXG}\arrow{r}{\piGH}& G\times H \arrow{dd}{\piG}\\
\hX  &\  & \ \\  
\ &\hX\times G\arrow[swap]{r}{\piG}\arrow{lu}{\pihX} & G
\end{tikzcd}
\end{equation*}
The reader may notice the minor ambiguity in the use of $\piG$ and $\pihX$.
All these projections commute with the respective actions of $T$ and $\hT$.
 On the level of the spaces $\MW$, $\MWG$, $\GMW$ and
$\GMWG$, this translates to
\begin{equation*}
\begin{tikzcd} 
\ &\GnuW \arrow[swap]{ld}{\pihX_*}\arrow{dd}{\pihXG_*}\arrow{r}{\piGH_*}& \nuW(\hX)\arrow{dd}{\piG_*}\\
\hX \arrow[dotted]{rru}{\nuW\;} \arrow[dotted]{rrd}{\;\nuWG}&\  & \ \\  
\ &\GnuWG\arrow[swap]{r}{\piG_*}\arrow{lu}{\pihX_*} & \nuWG(\hX)
\end{tikzcd}
\begin{tikzcd} 
\ &\GMW \arrow[swap]{ld}{\pihX_*}\arrow{dd}{\pihXG_*}\arrow{r}{\piGH_*}& \MW\arrow{dd}{\piG_*}\\
\hX \arrow[dotted]{rru}{\nuW\;} \arrow[dotted]{rrd}{\;\nuWG}&\  & \ \\  
\ &\GMWG\arrow[swap]{r}{\piG_*}\arrow{lu}{\pihX_*} & \MWG
\end{tikzcd}
\end{equation*}
where we use the standard notation $\pi_*(\nu)(A):=\nu(\pi^{-1}(A))$.
\begin{lemma}\label{lemma:diagrams-basic}
Each solid arrow in the r.h.s. diagram represents a factor between compact dynamical systems, that is the maps are continuous, surjective and commute with the actions $S$ respectively $\hT$. The maps $\nuW$ and $\nuWG$ represented by the dotted arrow are neither onto nor continuous, in general. Yet they commute with the other maps.
\end{lemma}
\begin{proof}
All arrows in the left hand diagram represent onto maps by definition, all solid arrows continuous maps.
As all closures involved in the right hand diagram are compact metric spaces, and as $\pihXG_*(\GnuW)=\GnuWG$, $\piGH_*(\GnuW)=\nuW(\hX)$, $\piG_*(\GnuWG)=\MWG$ and $\piG_*(\nuW(\hX))=\nuWG(\hX)$, also all solid arrows on that side represent onto maps, and as they all 
commute with the respective actions of $S$ and $\hT$, they all represent factor maps of compact dynamical systems.
\end{proof}

We will see in Lemma~\ref{lemma:proj-homeo-1} that $\pihXG_*:\GMW\to\GMWG$ is always a homeomorphism because of 
the injectivity of $\piG|_\LL$. Therefore the diagram simplifies to the one in Figure~\ref{figure:diagram}.
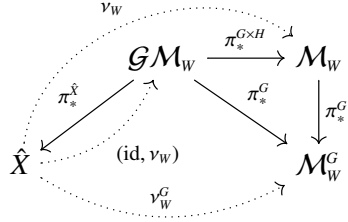
\begin{figure}[h]
\begin{equation*}
\begin{tikzcd} 
\ &\GMW\arrow{rd}{\piG_*}\arrow[swap]{ld}{\pihX_*}\arrow{r}{\piGH_*} & \MW\arrow{d}{\piG_*} \\  
\hX \arrow[dotted, bend right=30]{rr}{\nuWG}
\arrow[dotted, bend right=30,swap]{ru}{(\id,\nuW)} 
\arrow[dotted, bend left=60]{rru}{\nuW}&\ & \MWG
\end{tikzcd}
\end{equation*}
\caption{The dynamical systems under study\label{figure:diagram}}
\end{figure}
Observe also that 
\begin{itemize}
\item[-]
$\piGH_*:\GMW\to \MW$ is \oneone on $(\piGH_*)^{-1}(\MW\setminus\{\0\})$, where $\0$ denotes the zero configuration, and
\item[-]
if the window $W$ has non-empty interior, then $\piGH_*$ is a homeomorphism (Proposition~\ref{prop:projection-homeo}).
\end{itemize}
In particular, the map $\hat\pi:=\pihX_*\circ(\piGH_*)^{-1}:\MW\setminus\{\0\}\to\hX$, which associates to each non-zero configuration its torus parameter, is well defined, see Definition~\ref{def:hat-pi}.

The factor maps $\piG_*$ may have a more complicated structure, and for some results we will have to make the additional assumption that 
$\piG_*:\MW\to\MWG$ is \oneone, which implies immediately that $\piG_*$ is a homeomorphism. 
For windows $W\subset\R^d$ with $W=\overline{\inn(W)}$ we will show that this assumption is satisfied, so that the diagram from Figure~\ref{figure:diagram} simplifies further to
\begin{tikzcd} 
\hX &\GMW\arrow[leftrightarrow]{r}{\piG_*}\arrow[swap]{l}{\pihX_*} & \MWG\ .
\end{tikzcd}

For $\BB$-free systems and similar ones we will show that the restriction of $\piG_*$ to a rather large subset $\Y$ of $\MW$ (or $\GMW$) is \oneone. This set was identified before in \cite{Peckner2012}, see also \cite{Kulaga-Przymus2014, BaakeHuck14}.

The advantages of the general picture shown in the above diagrams are among others:
\begin{itemize}
\item[-]
The structure of the map $\pihX_*:\GMW\to\hX$ can be conveniently analysed under quite general assumptions so that the special properties of particular systems enter only the analysis of the map $\piG_*$.
\item[-]
The orbit closures of $\nuW(\hx)$ and $\nuWG(\hx)$ for different points $\hx\in\hX$ can all be studied as subsets of $\MW$ and $\MWG$, respectively.
Indeed, it will turn out that these orbit closures are identical for $m_\hX$-a.a. $\hx$.
\end{itemize}

\begin{remark}
In some publications on model sets such as \cite{BHP97, Moody2002}, the compact factor group $\hX$ is called a torus and denoted by $\T$. The set $\MWG$ contains all orbit closures $\X_{\scriptscriptstyle W'}$ of weak model sets produced from shifted copies
$W'$ of the window $W$.
\end{remark}

\section{Main results}\label{sec:main-results}

We start with some purely topological results about the dynamical system on (the closure of) the graph of the torus parametrisation and their consequences for identifying maximal equicontinuous factors and the lack of weak mixing. In a number of specific situations these consequences are well known, see e.g.~\cite{Robinson2007, BaakeHuck14, ABKL14}.

\subsection{Topological results}\label{sec:top}
Our main topological results will show 
that $(\hX,\hT)$ is the maximal equicontinuous factor of $(\GMW,S)$, and
how close to or far from being homeomorphisms $\piGH_*$ and $\piG_*$ are.
We will prove this starting with the observation that $\GMW$ is the closure of the graph of an upper semicontinuous function  which is invariant under the skew product action $\hT\times S$ on $\hX\times \cM$. To state the relevant proposition we make the following definitions:
\begin{definition}[Zeros and continuity points of $\nuW$]\quad
\begin{enumerate}[a)]
\item $\CW\subseteq \hX$ is the set of all continuity points of the map $\nuW$, i.e., the set of all points $\hx\in \hX$ such that for each sequence $(\hx^n)_n$ in $\hX$ which converges to $\hx$, the sequence of configurations $\nuW(\hx^n)$ 
converges vaguely to $\nuW(\hx)$.
\item $\ZW\subseteq\hat X$ is the set of all points $\hx\in\hX$ for which $\nuW(\hx)$ is the zero configuration that we denote by $\0$.
\end{enumerate}
\end{definition}

\begin{remark}\label{rem:CW}
\begin{enumerate}[a)]
\item  $\hx\in \CW$ if and only if $(\pihX_*)^{-1}\{\hx\}\cap \GMW=\{(\hx,\nuW(\hx))\}$.
\item $\CW$ is also the set of continuity points of $\nuWG=\piG_*\circ\nuW$ and
$Z_W=\{\hx\in\hX: \nuWG(\hx)=\0\}$. The latter identity is obvious; the first one follows from the fact that the graph closures $\GMW$ of $\nuW$ and $\GMWG$ of $\nuWG$ are fibre-wise homeomorphic, see Lemma~\ref{lemma:proj-homeo-1}.
\item We have $\hX=\ZW$ if and only if $W=\emptyset$.
\item The sets $\CW$ and $\ZW$ are $\hT$-invariant.
\end{enumerate}
\end{remark}

\begin{proposition}[Upper semicontinuity of $\nuW$]
\label{prop:projection-usc}\quad
\begin{enumerate}[a)]
\item 
Suppose that $\lim_{n\to\infty}\hx^n=\hx$. Then $\nu\leqslant\nuW(\hx)$ for all vague limit points $\nu$ of the sequence $(\nuW(\hx^n))_n$, and $d\nu/d\nuW(\hx)$ takes only values $0$ and $1$. (This is the upper semicontinuity of $\nuW$ w.r.t. the natural order relation on $\cM$.)
\item
$\hx\in \CW$ if and only if $\{\nu\in \MW: \nu\leqslant\nuW(\hx)\}=\{\nuW(\hx)\}$.
\item $\CW$ is a dense $G_\delta$-subset of $\hX$. 
\item If $\inn(W)=\emptyset$, then $\ZW=\CW$ is a dense $G_\delta$-subset of $\hX$. If in addition $W\ne\emptyset$, then $\0\in\overline{\nuW(\hX\setminus\ZW)}$.
\item If $\inn(W)\neq\emptyset$, then $\ZW=\emptyset$ and, even more, $\0\not\in \MW$.
\end{enumerate}
These assertions remain valid if $W$ is only assumed to be closed (and not compact).
\end{proposition}

\begin{remark}\label{remark:projection-lower-bound}
Proposition~\ref{prop:projection-usc} shows that
the set $\GMW$ is upper bounded by the graph of the upper semicontinuous function $\nuW:\hX\to\cM$. In a similar way it can be lower bounded by the lower semicontinuous function 
$\muW:=\nu_{H}-\nu_{\scriptscriptstyle\inn(W)^c}$. This function is lower semicontinuous, because
$\nu_H(\hx)=\sum_{\ell\in\LL}\delta_{x+\ell}$ is clearly continuous as a function of $\hx=x+\LL$, and as $\inn(W)^c$ is closed, $\nu_{\scriptscriptstyle \inn(W)^c}$ is upper semicontinuous.
Furthermore, as $\nuW+\nu_{\scriptscriptstyle\inn(W)^c}=\nu_{\scriptscriptstyle W\cup\inn(W)^c}+\nu_{\scriptscriptstyle W\cap\inn(W)^c}=\nu_H+\nu_{\scriptscriptstyle \partial W}$, 
we have $\muW=\nuW-\nu_{\scriptscriptstyle \partial W}\leqslant\nuW$. Hence, as $\GMW$ is the vague closure of $\cG\nuW$, this set is lower bounded by the graph of $\muW$. (For weak model sets with $\inn(W)=\emptyset$, we have $\muW( \hx)=\0$ for all $\hx\in\hX$.) Observe that for general $\nu\in \cM$ the condition $\muW(\hx)\leqslant\nu\leqslant\nuW(\hx)$ for some $\hx\in\hX$ does not imply $\nu\in\MW$. But see the
$\BB$-free systems from \cite{Kulaga-Przymus2014} discussed in Subsection~\ref{subsec:b-free} for a situation where this implication does hold.
\end{remark}

Some pieces of information on how to pass information about the system $(\GMW,S)$ to its various factor systems are collected in the following proposition.

\begin{proposition}\label{prop:projection-homeo}
\begin{enumerate}[a)]
\item $\piG_*:\MW\to\MWG$ is continuous.
\item $\piGH_*:\GMW\to \MW$ is \oneone on $(\piGH_*)^{-1}(\MW\setminus\{\0\})$, and
if $\inn(W)\neq\emptyset$, then $\piGH_*$ is a homeomorphism.
\item $\pihX_*:(\GMW,S)\to(\hX,\hT)$ is a topological almost \oneone extension of its maximal equicontinuous factor.
\end{enumerate}
\end{proposition}
Our first main result provides topological information on the various systems under study. Its proof, as well as those of the other results of this subsection, is provided in Sections~\ref{sec:basics} -- \ref{sec:projected}.

\begin{theorem}\label{theo:projection-top}(Topological factors and extensions)
\begin{enumerate}[a)]
\item If $\inn(W)\neq\emptyset$, i.e., if $\ZW=\emptyset$, then
\begin{equation}\label{eq:projection-a-1-1-ext1}
\begin{split}
\pihX_*&:(\GMW,S)\to(\hX,\hT)\quad\text{ and}\\
\pihX_*\circ(\piGH_*)^{-1}&:\hspace{3mm}(\MW,S)\to(\hX,\hT)
\end{split}
\end{equation}
are topological almost \oneone extensions of maximal equicontinuous factors.
In particular, none of the systems is topologically weakly mixing, if $\hX$ is non-trivial.

Furthermore, the restrictions of $S$ to the subsystems $\overline{\cG({\nuW}|_{\CW})}\subseteq \GMW$ and $\overline{\nuW(\CW)}\subseteq\MW$ are almost automorphic extensions\footnote{This means that both systems are minimal topological almost \oneone extensions of $(\hX,\hT)$.}
 of $(\hX,\hT)$, and 							they are the only minimal subsystems of $(\GMW,S)$ and $(\MW,S)$, respectively.
\item If $\inn(W)=\emptyset$, i.e., if $\ZW\neq\emptyset$, 
then $\pihX_*:(\GMW,S)\to(\hX,\hT)$ is a topological almost \oneone extension of its maximal equicontinuous factor,
whereas $(\MW,S)$ and $(\MWG,S)$ have no nontrivial minimal equicontinuous factor. On the other hand, $(\MW,S)$ and $(\MWG,S)$ are never topologically weakly mixing, except if $\card(\MWG)=1$.
\end{enumerate}
\end{theorem}
\begin{corollary}\label{cor:projection-top}
\begin{enumerate}[a)]
\item All assertions of the theorem remain valid for the system $(\GMWG,S)$
(see Lemma~\ref{lemma:proj-homeo-1}), and
if $\piG_*:\MW\to\MWG$ is \oneone\, (and hence a homeomorphism), the same is true for $(\MWG,S)$.
\item
If $\inn(W)\neq\emptyset$ and 
if the window $W':=\overline{\inn(W)}$ is \emph{aperiodic}, i.e. if $h+W'=W'$ implies $h=0$, then the restriction of $S$ to the subsystem  $\overline{\nuWGprime(\CWprime)}=\overline{\nuWG(\CW)}\subseteq\MWG$ is an almost automorphic extension of $(\hX,\hT)$, and it is the only minimal subsystem of $(\MWG,S)$.
\end{enumerate}
\end{corollary}
\noindent While the first part of this corollary is obvious, the second part is proved in
Section~\ref{sec:projected}.
How to get rid of the aperiodicity assumption is discussed in Remark~\ref{rem:aperiodicity}.

\subsection{Measure theoretic results}\label{sec:mt}
On the measure theoretic side we have results related to the Mirsky measure
\cite{Abdalaoui2014,CS13,Kulaga-Przymus2014,Sarnak2011} and its spectral properties, to configurations with maximal density \cite{Moody2002,S14}
and to the denseness of individual orbits.

\begin{proposition}\label{prop:projection-mth}(Haar measure of $\ZW$ and $\CW$)

\begin{enumerate}[a)]
\item The sets $\CW$ and $\ZW$ have either full or vanishing Haar measure.
\item  $m_\hX(\ZW)=0$ if and only if  $m_H(W)>0$.
\item $m_\hX(\CW)=1$ if and only if  $m_H(\partial W)=0$. 
\end{enumerate}
\end{proposition}

\begin{definition}\label{def:Q_M}(Mirsky measures)
Denote by 
\begin{compactenum}[-]
\item $\QM:=m_\hX\circ\nuW^{-1}$ the lift of Haar measure on $\hX$ to  $\MW$
\item $\QGM:=m_\hX\circ(\id_\hX,\nuW)^{-1}$ the lift of Haar measure on $\hX$ to $\GMW$
\item $\QMG:=m_\hX\circ(\nuWG)^{-1}$ the lift of Haar measure on $\hX$ to  $\MWG$
\item $\QGMG:=m_\hX\circ(\id_\hX,\nuWG)^{-1}$ the lift of Haar measure on $\hX$ to $\GMWG$
\end{compactenum}

\end{definition}
\begin{remark}
In some cases, as e.g.~the square-free integers and the $\BB$-free systems \cite{Kulaga-Przymus2014}, the measure $\QMG$ is called \emph{Mirsky measure}. Note also that in the trivial case where $m_H(W)=0$, the measures $\QM$ and $\QMG$ are the point masses in the zero-configuration $\0$.
\end{remark}

\begin{theorem}\label{theo:projection-mth}(Measurable factors and extensions)\quad Consider the case $m_H(W)>0$.
\begin{enumerate}[a)]
\item The systems $(\hX,m_\hX,\hT)$, $(\GMW,\QGM,S)$ and $(\MW,\QM,S)$ are measure theoretically isomorphic, and the system $(\MWG,\QMG,S)$ is a factor of $(\hX,m_\hX,\hT)$ via the measurable factor map $\nuWG$. 
In particular, $(\MWG,\QMG,S)$ has pure point dynamical spectrum\footnote{In ergodic theory, the term \emph{discrete spectrum} may be more common, see e.g.~\cite{Walters1982}.}, and its group of eigenvalues is a subgroup of the group of eigenvalues of $(\hX,m_\hX,\hT)$.

\item If $\nuWG:\hX\to\MWG$ is not constant $m_\hX$-a.e., then $(\hX\times\MWG,\QGM,\hT\times S)$ is a nontrivial joining of the systems $(\hX,m_\hX,\hT)$ and  $(\MWG,\QMG,S)$.

\item If $m_H(\partial W)=0$, then  $\ZW=\emptyset$ and $\QGM$, $\QM$ and $\QMG$ are the only 
invariant measures for $(\GMW,\hT\times S)$, $(\MW,S)$ and $(\MWG,S)$, respectively.
In particular, these systems are uniquely ergodic.

\end{enumerate}
\end{theorem}
\noindent These results are proved in Section~\ref{sec:projected}.
\begin{remark}\label{rem:ppdiff-dyn}
Baake and Lenz \cite[Theorem 7]{BaakeLenz2004} proved in a very general context, which encompasses the setting chosen in this paper, that (with our specialised notation) the diffraction measure of $\QMG$ is purely atomic if and only if the dynamical system $(\MWG,\QMG,S)$ has pure point spectrum. But the latter property is guaranteed 
 by Theorem~\ref{theo:projection-mth}a for
all weak model sets with compact window of positive Haar measure.
\end{remark}


\begin{remark}(Entropy)
An interesting question concerns the topological entropy of the above systems in the case $m_H(\partial W)>0$. It is bounded by $\dL\cdot m_H(\partial W)\cdot\log2$, as may be proved using relative entropy theory for skew product systems. Recently, a combinatorial proof for the corresponding pattern entropy bound has been given in \cite{HuckRichard15}.  More generally, if $m_H(\partial W)>0$, the above systems may have many invariant probability measures, all extending the Haar measure on $\hX$. They all have entropy at most $\dL\cdot m_H(\partial W)\cdot\log2$.

For the example of $\BB$-free systems, which is discussed below, the upper bound is attained. This follows from the hereditary property of that system, i.e., for any given configuration $\nu\in \MW$ any of its subconfigurations $\nu'\le\nu$ is also in $\MW$. For this particular example, its simplex of invariant probability measures is quite well understood \cite{Peckner2012, Kulaga-Przymus2014}.
\end{remark}

A certain kind of sequences $(A_n)_{n\in\N}$ of compact subsets of $G$ that typically exhaust $G$ and  have ``nice'' boundaries  are called \emph{tempered van Hove sequences}. They are
discussed by Moody in \cite{Moody2002}.\footnote{(Generalised) van Hove sequences were introduced by Schlottmann \cite{Schlottmann00}, where also their existence is discussed. For compact $A,K\subseteq G$, let $\partial^KA=((K+A)\setminus\inn(A))\cup ((-K+\overline{G\setminus A})\cap A)$ denote the $K$-boundary of $A$. A sequence $(A_n)_n$ of compact subsets of $G$ of positive Haar measure is van Hove if $\lim_{n\to\infty} m_G(\partial^KA_n)/m_G(A_n)=0$ for any compact $K\subseteq G$. The sequence $(A_n)_n$ is tempered \cite{Lindenstrauss2001} if there is a constant $C>0$ such that $m_G(\bigcup_{k<n}(A_k-A_k))\le C m_G(A_n)$ for every $n$. For further background see \cite{BaakeLenz2004, LenzRichard07, mr13} and references therein.} Here it suffices for the moment to note that $A_n=[-n,n]^d$ defines tempered van Hove sequences in $\Z^d$ and in $\R^d$, and that tempered van Hove sequences exist in every $\sigma$-compact locally compact abelian group.

\begin{theorem}(Moody \cite[Theorem 1]{Moody2002}\label{theo:Moody})\quad\\
Let $(A_n)_{n\in\N}$ be a tempered van Hove sequence of subsets of $G$. Then
\begin{equation*}
\lim_{n\to\infty}\frac{\nuW(\hx)(A_n\times H)}{m_G(A_n)}
=
\dL\cdot m_H(W)\quad\text{for $m_\hX$-a.e. $\hx\in\hX$.}
\end{equation*}
As $W$ is compact, the convergence is indeed semi-uniform: For each $\epsilon>0$ there is $n_0\in\N$ such that for all $n\geqslant n_0$ and all $\hx\in\hX$
\begin{equation}\label{eq:nu-density}
\frac{\nuW(\hx)(A_n\times H)}{m_G(A_n)}
\leqslant
\dL\cdot m_H(W) +\epsilon\ .
\end{equation}
\end{theorem}
\begin{proof}
The a.e. convergence is stated explicitly (and proved) in \cite[Theorem 1]{Moody2002}, with a particular normalisation of the Haar measure $m_{G\times H}$ such that $\dL=1$. The argument is based on a version of the generalised Birkhoff ergodic theorem, which may be replaced by \cite{Lindenstrauss2001}.
At the end of that proof additional arguments are provided to prove uniform convergence when $m_H(\partial W)=0$. The same arguments yield the semi-uniform convergence for general compact windows $W$, see also \cite{Sturman2000,mr13}.
\end{proof}
\addtocounter{corollary}{1}
\begin{corollary}
The same statements hold for the ratios $\frac{\nuWG(\hx)(A_n)}{m_G(A_n)}$, because
$\nuWG(\hx)(A_n)=\nuW(\hx)(A_n\times H)$.
\end{corollary}

We extend Theorem~\ref{theo:Moody} to cover also measures $\nu$ dominated by $\nuW(\hx)$. 

\begin{theorem}\label{theo:Moody-extension}(Ergodic point densities)\\
Let $(A_n)_{n\in\N}$ be a tempered van Hove sequence of subsets of $G$ as before.
Let $P$ be an ergodic $S$-invariant probability measure on $\MW$. Then, for $P$-a.e. $\nu\in \MW$, the inequality
\begin{equation*}
D_P:=\lim_{n\to\infty}\frac{\nu(A_n\times H)}{m_G(A_n)}
\quad\text{exists and is }\leqslant \dL\cdot m_H(W)
\end{equation*}
is satisfied,
with equality $P$-a.e. if and only if $P=\QM$.  So in particular, $D_{\QM}=\dL\cdot m_H(W)$.

For each sufficiently small compact zero neighbourhood $B\subseteq G$ holds
\begin{equation}\label{eq:D_P-formula}
D_P=\frac{1}{m_G(B)}\int_{\MW}\nu(B\times H)\,dP(\nu)\ ,
\end{equation}
so that in particular the map $P\mapsto D_P$ is upper semicontinuous w.r.t. to the weak topology on probability measures on $\MW$.
\end{theorem}
\noindent The proof of this theorem, which follows closely the proof of Moody's theorem,  is given in Section~\ref{sec:Moody-proof}.
\begin{corollary}\label{coro:Moody-extension}
Also the assertions of Theorem~\ref{theo:Moody-extension} remain valid if $\MW$, $\QM$ and $\nu(\,.\,\times H)$ are replaced by $\MWG$, $\QMG$ and $\nu(\,.\,)$, respectively. 
\end{corollary}
\noindent
This is obvious for measures $P$ on $\MWG$ that can be represented as $\tilde{P}\circ(\piG_*)^{-1}$ with some ergodic $S$-invariant probability $\tilde{P}$ on $\MW$. 
That all ergodic $S$-invariant $P$ on $\MWG$ can be represented in this way is proved in Section~\ref{sec:Moody-proof}.

For the rest of this section we fix one tempered van Hove sequence $(A_n)_{n\in\N}$. Our final theorem highlights the exceptional role played by configurations 
$\nuW(\hx)$
with maximal density with respect to $(A_n)_{n\in\N}$, i.e., by configurations $\nuW(\hx)$ with
\begin{equation}
\hx\in \hXmax
:=
\left\{\hx\in\hX: \lim_{n\to\infty}\nuW(\hx)(A_n\times H)/m_G(A_n)=D_{\QM}\right\}.
\end{equation}
These points occur already in \cite{Moody2002}, and their importance was stressed more recently by Nicolae Strungaru \cite{S14}. Here we show that the set
$\nuW(\hXmax)\subseteq\MW$ coincides with the set of generic points for the $S$-invariant probability measure $\QM$ on $\MW$. Observe before that
\begin{equation}
\hXmax=\hXmaxG:=
\left\{\hx\in\hX: \lim_{n\to\infty}\nuWG(\hx)(A_n)/m_G(A_n)=D_{\QM}\right\}.
\end{equation}

\begin{theorem} \label{theo:Moody-topogical}(Configurations of maximal density are generic for the Mirsky measure)
\begin{enumerate}[a)]
\item  $\hXmax$ is $\hT$-invariant and
$m_\hX(\hXmax)=1$.
\item
For each $\hx\in\hXmax$ the empirical measures 
\begin{equation*}
Q_{n,\hx}
:=
\frac{1}{m_G(A_n)}\int_{A_n}\delta_{S_{g}\nuW(\hx)}\,dm_G(g)
\end{equation*}
converge weakly to $\QM$. 
\item 
For $\hx\in\hX$ denote by $\MW(\hx)$ the orbit closure of $\nuW(\hx)$ under the action of $S$ in $\MW$.
Then $\QM (\MW(\hx))=1$ for all $\hx\in\hXmax$, i.e.,
$\supp(\QM)\subseteq\MW(\hx)\subseteq\MW$ for all $\hx\in\hXmax$.
\item 
$\supp(\QM)=\MW(\hx)$ for all $\hx\in\hXmax\cap\nuW^{-1}(\supp(\QM))$, and $m_\hX(\hXmax\cap\nuW^{-1}(\supp(\QM)))=1$.
\item $\overline{\nuW(\CW)}\subseteq\supp(\QM)$.
\end{enumerate}
\end{theorem}
\begin{corollary}
All assertions of this theorem remain valid if $\MW, \hXmax, \nuW(\hx), \QM$ and $\MW(\hx)$ are replaced by the corresponding objects
$\MWG, \hXmaxG, \nuWG(\hx), \QMG$ and $\MWG(\hx)$, respectively. For a) and b) this is obvious, for c) and d) one notes that $\piG_*(\MW(\hx))=\MWG(\hx)$ 
and $\piG_*(\supp(\QM))=\supp(\QMG)$, because the involved spaces are compact.
Finally, e) follows from Remark~\ref{rem:CW}.
\end{corollary}

Combining Theorem~\ref{theo:projection-mth}a), Theorem~\ref{theo:Moody-topogical}c) and the previous corollary, we arrive at the following result.

\begin{corollary}\label{cor:pp-spectrum}
Consider the case $m_H(W)>0$. For any $\hx\in \hX_{max}$, the systems $(\hX,m_\hX,\hT)$ and $(\MW(\hx),\QM,S)$ are measure theoretically isomorphic, and the system $(\MWG(\hx),\QMG,S)$ is a factor of $(\hX,m_\hX,\hT)$ via the factor map $\nuWG$. 
In particular, $(\MWG(\hx),\QMG,S)$ has pure point dynamical spectrum, and its group of eigenvalues is a subgroup of the group of eigenvalues of $(\hX,m_\hX,\hT)$.
\end{corollary}

\begin{remark}\label{remark:support}
In view of Theorem~\ref{theo:Moody-topogical}e it may be worth noting that the two sets
$\overline{\nuW(\CW)}$ and $\supp(\QM)$ are the topological respectively measure theoretic result of the process to rid the set $\MW$ of ``negligible'' parts.
Indeed, as the topology on $\MW$ has a countable base and as $\CW$ is a dense $G_\delta$-subset of $\hX$, it is not hard to show that
\begin{compactenum}[a)]
\item $\overline{\nuW(\CW)}$ is the intersection of all sets $\overline{\nuW(R)}$ where $R$ ranges over all dense $G_\delta$-subsets of $\hX$, 
\item $\supp(\QM)$ is the intersection of all sets $\overline{\nuW(F)}$ where $F$ ranges over all full measure subsets of $\hX$,
\item $\overline{\nuW(\CW)}=\supp(\QM)$ if $m_H(\partial W)=0$.
\end{compactenum}
The same statements hold for $\nuWG$ and $\QMG$.
The proofs are provided in Section~\ref{sec:Moody-proof}.
\end{remark}

The next corollary is, in conjunction with Theorem~\ref{theo:projection-top}a and Theorem~\ref{theo:projection-mth}c, an immediate consequence of the previous remark.

\begin{corollary}(Strict ergodicity)\label{coro:strict}\\
Whenever $m_H(\partial W)=0$ and $\inn(W)\ne\emptyset$, the restrictions of $S$ to $\supp(\QM)$ and to $\supp(\QMG)$ are strictly ergodic, i.e., minimal and uniquely ergodic.
\end{corollary}

\begin{remark}\label{remark:patterns}(Pattern frequencies)\\
Let us call any non-empty, finite $\LL_1\subset \LL$ a local configuration or pattern. We are interested in how frequently shifted copies of a given pattern appear in some fixed configuration $\nu\in \MW$. For some van Hove sequence $(A_n)_{n\in\N}$ in $G$, we may thus consider relative frequencies
\begin{displaymath}
\begin{split}
f_n(\LL_1,\nu)
&=
\frac{\card\left\{(g, h) \in A_n\times H: -(g,h) + \LL_1\subseteq \supp(\nu)\right\}}{m_G(A_n)}
\end{split}
\end{displaymath}
and ask whether a limiting frequency $f(\LL_1,\nu)$ exists as $n\to\infty$. 
As a corollary to Theorem~\ref{theo:Moody-topogical}b, we prove in Section~\ref{sec:Moody-proof} that for configurations $\nuW(\hx)$ with maximal density  this is indeed the case, namely
\begin{equation}\label{form:pf}
\lim_{n\to\infty} f_n(\LL_1,\nuW(\hx))
=
f(\LL_1,\nuW(\hx))
=
\dL\cdot m_H\left(\bigcap_{\ell\in \LL_1}(W-\ell_H)\right)\quad\text{for all $\hx\in\hXmax$.}
\end{equation}
As $\piG_*|_\LL$ is \oneone, it is easily checked that the projected configurations in $\MWG$ have the same pattern frequencies, i.e.,
\begin{equation*}
f_n(\LL_1,\nu)
=
\frac{\card\left\{g\in A_n: -g+ \piG(\LL_1)\subseteq \supp(\piG_*\nu)\right\}}{m_G(A_n)}\ .
\end{equation*}

\end{remark}

\subsection{Remarks and comments}

\begin{remark}(Generic windows and repetitive model sets)\label{rem:generic-window}
\\
It follows from a characterisation of the set $\CW$ in Lemma~\ref{lemma:projection-CW}, that a point $\hx=(x+\LL)$ belongs to $\CW$ if and only if $\piH(\LL)\cap((\partial W)-x_H)=\emptyset$, i.e., if and only if the window $W-x_H$ is generic in the sense of Schlottmann \cite{Schlottmann1998}. If $\inn(W)\ne\emptyset$, then Theorem~\ref{theo:projection-top}a  shows that the subsystem $\overline{\nuW(\CW)}$ is almost automorphic, in particular minimal. Observe $\overline{\nuW(\CW)}=\MW(\hx)$ and hence also $\overline{\nuWG(\CW)}=\MWG(\hx)$
for each point $\hx\in \CW$. But minimality is equivalent to almost periodicity which translates to repetitivity in the traditional language of model sets. Hence we rediscover the fact that generic windows with non-empty interior generate repetitive model sets/configurations. For topologically regular windows, this was also proved by Robinson \cite[Prop. 5.18 and Cor. 5.20]{Robinson2007}.  Note, however, that when $m_H(\partial W)>0$, then $m_\hX(\CW)=0$ so that the set of these configurations has Mirsky measure zero.
\end{remark}

\begin{remark}(Some open problems related to $\supp(\QM)$)\\
Theorem~\ref{theo:Moody-topogical} suggests that the ``relevant'' dynamical system to look at is not $(\MW,S)$ but $(\supp(\QM ),S)$, because it has the property that the orbit of $\nuW(\hx)$ is dense in this space for all $\hx\in\hXmax\cap\nuW^{-1}(\supp(\QM))$. This imposes the following questions:
\begin{compactenum}[i)]
\item Which points belong to $\MW\setminus\supp(\QM)$\,?\\ 
A very partial answer is that $\0\in\supp(\QM)$ whenever $\0$ belongs to $\MW$ at all (combine Proposition~\ref{prop:projection-usc}d,e with Theorem~\ref{theo:Moody-topogical}c).
\item Which points belong to $\nuW(\hXmax)\setminus\supp(\QM)$\,? \\
Observe that if $\hx\in\hXmax$ but $\nuW(\hx)\not\in\supp(\QM)$, then
there is a (locally specified) open neighbourhood $U$ of $\nuW(\hx)$ in $\MW$ such that
$\QM(U)=0$. As $\nuW(\hx)$ is generic for $\QM$, it follows that the local pattern of $\nuW(\hx)$ specified by $U$ can repeat in $\nuW(\hx)$ only with density zero. 
\end{compactenum}
These questions will be studied for the examples in Subsection~\ref{subsec:interval-windows}. The following proposition indicates that the general case is subtle. A proof can be found at the end of Section~\ref{sec:Moody-proof}.

\begin{proposition}\label{prop:exceptional-points}
Let $\hx=(x_G,x_H)+\LL\in\hX$. Then $\nuW(\hx)\in\supp(\QM)$ if and only if
\begin{equation}\label{eq:non-exceptional}
m_H\left\{h\in H: (-x_H+W)\cap\piH(\LL')=(-h+W)\cap\piH(\LL')\right\}>0
\quad\text{for all finite $\LL'\subseteq\LL$.}
\end{equation}
\end{proposition}
\noindent
\end{remark}

\begin{remark}(Results for non-compact windows)\label{rem:rc}\\
Suppose that $W_0\subseteq H$ is only relatively compact and that $\partial W_0$ is nowhere dense.  This setting encompasses the Fibonacci chain, 
see Example~\ref{ex:Fibonacci}. It
extends slightly the approaches \cite{Schlottmann00, FHK2002, Robinson2007, BLM07}, which restrict to topologically regular windows for most of their results - a stronger assumption than our $\inn(\partial W_0)=\emptyset$. 
Let $W=\overline{W_0}$. Then $\partial W\subseteq\partial W_0$ and
$\0\leqslant\nuW-\nuWzero\leqslant\nuWbdzero$, where $\ZWbdzero=\CWbdzero$ is a dense $G_\delta$-set by Proposition~\ref{prop:projection-usc}.
\\[1mm]
\emph{Topological results:}\;
It follows that $\CWzero\cap\CWbdzero=\CW\cap\CWbdzero$,
so that
the set of continuity points of $\nuWzero$ contains a dense $G_\delta$-set. Moreover, $\overline{\nuWzero(\CWzero)}=\overline{\nuWzero(\CWbdzero)}=
\overline{\nuW(\CWbdzero)}=\overline{\nuW(\CW)}$ and similarly also $\overline{\nuWGzero(\CWzero)}=
\overline{\nuWG(\CW)}$, compare Remark~\ref{remark:support}.
This shows also that Remark~\ref{rem:generic-window} holds for $W_0$.

In fact, also our main Theorem~\ref{theo:projection-top} and its Corollary~\ref{cor:projection-top} remain true, with $W_0$ replacing $W$. 
(For checking the proofs one should observe that for each $\nu\in\MWzero\setminus\{\0\}$ there is a unique $\hx\in\hX$ such that $\nu\leqslant\nuW(\hx)$, although $\nu\leqslant\nuWzero(\hx)$ need not hold. For Corollary~\ref{cor:projection-top}b note also that $W'=\overline{\inn(\overline{W_0})}=\overline{\inn({W_0})}$.)
A sufficient criterion for $\piG_*:\MWzero\to\MWGzero$ to be \oneone  is that $\overline{W_0}$ is aperiodic and that $W_0$ is topologically regular, i.e., $\overline{W_0}=\overline{\inn(W_0)}$, the latter condition implying that $\partial W_0$ is nowhere dense. 
For its proof, note that aperiodicity of $\overline{W_0}$ implies that $(W_0, \LL)$ is strongly uniquely coding, compare the proof of Lemma~\ref{lemma:regular-window}, and that Lemma~\ref{lemma:projection-injective-1} remains valid, with $W_0$ relatively compact replacing $W$. Of course, $\nuWzero$ will generally no longer be upper semicontinuous. 
\\[1mm]
\emph{Measure theoretic results:}\;
Assume that $W_0\subseteq H$ is relatively compact and measurable. Then the map $\nuWzero: \hX\to\MWzero$ is measurable, which is checked with methods from the proof of Remark~\ref{remark:patterns}. 
As a consequence, the resulting Mirsky measures in Definition~\ref{def:Q_M} are well defined, and Theorem~\ref{theo:projection-mth}, Corollary~\ref{coro:strict} and Proposition~\ref{prop:exceptional-points} hold, all with $W_0$ replacing $W$. Denote the resulting Mirsky measure on $\cM$ by $\QM^0$ and fix a tempered van Hove sequence $(A_n)_n$. As also the first part of Moody's Theorem~\ref{theo:Moody} holds,  there is a full  measure set $\hX_1\subseteq \hX$ such that configurations are generic for $\QM^0$. These configurations are then pure point diffractive,  as follows from Theorem~\ref{theo:projection-mth}, Remark~\ref{rem:ppdiff-dyn}, and the ergodic theorem. (In \cite[Cor.~1]{Moody2002}, pure point diffractivity is shown in a different way.) Since $\QM^0$-almost sure density might be less than maximal density $D_{\QM}=\mathrm{dens}(\LL)\cdot m_H(W)$, Theorem~\ref{theo:Moody-topogical} will no longer be  valid in general. An example having an open window with fat Cantor set boundary has been discussed in \cite{Moody2002}. 

In order to recover the setting and results of \cite{BaakeHuckStrungaru15}, let us compare configurations from $\MWzero$ to configurations from $\MW$. For any $\hx\in\hX$ we have $0\leqslant\nuWzero(\hx)\leqslant\nuW(\hx)$, so in particular the density of $\nuWzero(\hx)$ along the sequence $(A_n)_n$ is bounded by that of $\nuW(\hx)$ and hence by $D_{\QM}$.
If $\nuWzero(\hx)$ achieves the maximal density $D_{\QM}$, then $\hx\in\hXmax$ and the density of $\nuW(\hx)-\nuWzero(\hx)$ is clearly zero. Consequently, the empirical measures
$Q^0_{n,\hx}
:=
\frac{1}{m_G(A_n)}\int_{A_n}\delta_{S_{g}\nuWzero(\hx)}\,dm_G(g)
$ are asymptotically equivalent to the measures $Q_{n,\hx}$ in the sense that both sequences do have the same weak limit points, and Theorem~\ref{theo:Moody-topogical} 
implies that the measures $Q^0_{n,\hx}$ converge weakly to $\QM$. 
It follows that statistical properties of $\nuWzero(\hx)$ and $\nuW(\hx)$, like pattern frequencies and especially their autocorrelation coefficients, coincide for such $\hx$, compare \cite{BaakeMoody2004}.  Observe, however, that such $\nuWzero(\hx)$ need not be an element of $\MW$.
\\[1mm]
\emph{Combined results:}\;
If $m_H(\partial W_0)=0$, then $\partial W_0$ is nowhere dense and $0\leqslant\nuW-\nuWzero\leqslant\nuWbdzero$ implies
$\nuWzero=\nuW$ $m_\hX$-a.e. and on a dense $G_\delta$-set. Hence the corresponding Mirsky measures $\QM^0$ and $\QM$ coincide,
and the measures $Q^0_{n,\hx}$ converge weakly to $\QM^0$
for $\QM^0$-a.e. $\hx$. Moreover, $\supp(\QM^0)=\overline{\nuWzero(\CWzero)}$ and
$\supp(\QMG^0)=\overline{\nuWGzero(\CWzero)}$, and if $\inn(W_0)\ne\emptyset$, then both subsystems are strictly ergodic, see Corollary~\ref{coro:strict}.

Note that these results apply in particular to various subclasses of  so-called inter model sets, which are discussed in \cite{ABKL14}.
\end{remark}

\begin{remark}(The Mirsky measure as a zero temperature limit)\\
The measure $\QMG $ is characterised by the variational formula
\begin{equation*}
D_{\QMG}=\sup_P D_P\ ,\text{ where }
D_P=\int_{\MWG}\chi_B\,dP\text{\; with \;}\chi_B(\nu)=\frac{\nu(B)}{m_G(B)}\ .
\end{equation*}  
Here $B$ is any sufficiently small compact zero neighbourhood in $G$, and
the supremum extends over all $S$-invariant probability measures on $\MWG$, see Corollary~\ref{coro:Moody-extension}. As $B$ is compact, $\chi_B:\MWG\to\R$ is upper semicontinuous, and as any sufficiently small compact zero neighbourhood can be used, one can replace the indicator function of $B$ actually by a continuous approximation.
In this sense, $\QMG$ is the unique maximising measure for each such approximation. In the special case of $\Z$-actions there is some general theory for such measures, that asserts among others that for generic continuous observables there is a unique maximising measure and that this measure has zero entropy, and also that each ergodic measure is maximising for some continuous function, see  e.g. \cite{Jenkinson2006,Bremont2008}.

When there is even a unique maximising measure, then a very brief argument shows that this measure is the temperature zero limit of equilibrium states associated to the observable. The prerequisit for the argument is a suitable version of the thermodynamic formalism - in particular the existence of equilibrium states.\footnote{When $G=\Z^d$, this is covered by \cite{Misiurewicz1976}, see also \cite{Keller1997b}, and for general discrete abelian groups \cite{Bufetov2011} could be used.} Whenever this is available one can argue as follows: For an $S$-invariant probability measure $P$ on $\MWG$  denote by $h_S(P)$ the Kolgomorov-Sinai entropy of the dynamical system $(\MWG,S,P)$. 
Then for each $\beta>0$ there is at least one $S$-invariant probability measure $P_\beta$ that maximises the functional $P\mapsto h_S(P)+\beta\cdot\int\chi_B(\nu)\,dP(\nu)$. (In other words: $P_\beta$ is an equilibrium state for $\beta\,\chi_B$.) Let $Q$ be any weak limit of such measures $P_\beta$ along a sequence $\beta_n\to\infty$. Then
we have for each $S$-invariant probability measure $P$ on $\MWG$
\begin{equation*}
\begin{split}
\int\chi_B\,dP
&=
\lim_{n\to\infty}
\frac{1}{\beta_n}\left(h_S(P)+\beta_n\int\chi_B\,dP\right)
\leqslant
\limsup_{n\to\infty}
\frac{1}{\beta_n}\left(h_S({P_{\beta_n}})+\beta_n\int\chi_B\,dP_{\beta_n}\right)\\
&=
\limsup_{n\to\infty}\int\chi_B\,dP_{\beta_n}
\leqslant
\int\chi_B\,dQ\ ,
\end{split}
\end{equation*}
because $h_S({P_{\beta^n}})$ is bounded by the (finite!) topological entropy of $(\MWG,S)$.
This means that $Q$ maximises $\chi_B$, and as $\QMG$ is the only such measure, we conclude that $Q=\QMG$ is the weak limit of the $P_{\beta}$ as $\beta\to\infty$.
When the maximising measure is not unique, one can show (with a similar reasoning) that all temperature zero limit measures have maximal entropy among the maximising measures.
\end{remark}

\section{Examples}\label{sec:examples}

In this section we first illustrate our main results with the golden and silver mean chains, 
 continue by providing some general facts about topologically regular windows, and we finally discuss how also $\BB$-free systems fit our framework.

\subsection{Interval windows}\label{subsec:interval-windows}

One of the most elementary non-trivial settings for model sets is probably the case where $G=H=\R$, $W=[\alpha,\beta]$ is a compact interval, and where the lattice $\LL\subset \R^2$ is spanned by two vectors $v=(v_G,v_H)$ and $w=(w_G,w_H)$. Two classical examples are the golden (Fibonacci) and the silver mean chain, which are discussed in some detail in the monograph \cite{BaakeGrimm13}. Before we look at their peculiarities, we collect some facts that apply to both of them. 
Observe first that $\LL=\{mv+nw: m,n\in\Z\}$. From now on we assume that $v_G$ and $w_G$ are rationally independent and that the same is true for $v_H$ and $w_H$. 
Then $\piG|_\LL$ is \oneone, and $\piH(\LL)$ is dense in $H=\R$.

Note next that $\piG_*$ is a homeomorphism (see Lemma~\ref{example:projection-ex1}), so that we can restrict our discussion to the set $\MW$ and need not consider $\MWG$ and its subsystems. Further general facts are that $(\MW,S)$ is uniquely ergodic by Theorem~\ref{theo:projection-mth}c, because $m_H(\partial W)=0$, and that its subsystem $\overline{\nuW(\CW)}=\supp(\QM)$ is minimal
by Remark~\ref{remark:support} and Theorem~\ref{theo:projection-top}a. Because of the unique ergodicity, all points $\nu\in\MW$ are generic for $\QM$ and, in particular,
$\nuW(\hXmax)=\MW$.

In order to determine $\supp(\QM)$, we have a closer look at the set $\CW$ of continuity points of $\nuW$:
it follows from Lemma~\ref{lemma:projection-CW} that $\hx=x+\LL\not\in\CW$ if and only if
$\alpha\in x_H+ \piH(\LL)$ or $\beta\in x_H+ \piH(\LL)$. 
We distinguish the following two cases:
\begin{compactenum}[I)]
\item $\beta-\alpha\not\in\piH(\LL)$. Then, for each $x_H$, at most one of the points $\alpha$ and $\beta$ can belong to $x_H+\piH(\LL)$, and Proposition~\ref{prop:exceptional-points} implies at once that $\hx\in\supp(\QM)$. Hence $\MW=\supp(\QM)$, so that $(\MW,S)$ is minimal, i.e., it  coincides with the orbit closures of all its points. 
\item $\beta-\alpha\in\piH(\LL)$. Then, for each $x_H$, the point $\alpha$ belongs to $x_H+\piH(\LL)$ if and only if also $\beta$ belongs to this set. Hence Proposition~\ref{prop:exceptional-points} implies that those $\hx$ for which this happens do not belong to $\supp(\QM)$. 
Indeed, a moment's reflection shows that $\MW\setminus\supp(\QM)$ is a translate of $\nuW\left(\LL+(G\times\{0\})\right)$.

Traditionally one is mostly interested in $\hx=\hat0$, more precisely in $\nuW(\hat0)$ and
its orbit closure $\MW(\hat0)$. (Note that the support of $\nuWG(\hat0)=\piG_*(\nuW(\hat0))$ is the point set which is often denoted by $\oplam(W)$ in the literature.) 
The following two classical examples show that $\nuW(\hat0)$ can be contained in $\supp(\QM)$ or not.
\end{compactenum}

\begin{example}(Silver mean chain, see \cite[Section 7.1]{BaakeGrimm13})\quad\\
Let $v=(1,1)$, $w=(\sqrt{2},-\sqrt{2})$ and $W=[\alpha,\beta]=[-\frac{\sqrt 2}{2},\frac{\sqrt 2}{2}]$.
Then $\beta-\alpha=\sqrt2\in\piH(\LL)=\Z+\sqrt2\,\Z$, but $\alpha,\beta\not\in\piH(\LL)$ so that $\hat0\in \CW$ and therefore $\MW(\hat0)=\supp(\QM)$.
\end{example}

\begin{example}(Golden mean (Fibonacci) chain, see \cite[Example 7.3]{BaakeGrimm13})
\label{ex:Fibonacci}\quad\\
Let $v=(1,1)$ and $w=(\tau,\tau')$ where $\tau=\frac{1+\sqrt5}{2}$ and $\tau'=\frac{1-\sqrt5}{2}$.
Let $W=[\alpha,\beta]=[-1,-\tau']$. Then $\beta-\alpha=1-\tau'\in\piH(\LL)=\Z+\tau'\Z$ and 
$\alpha,\beta\in\Z+\tau'\Z$, so that $\hat0\not\in \CW$ and $\nuW(\hat0)\not\in\supp(\QM)$. This problem is sometimes addressed by considering half-open windows $W'$, for which $\hx=\hat0$ is still a point of discontinuity of $\nu_{{\scriptscriptstyle W'}}$, but for which $\nu_{{\scriptscriptstyle W'}}(\hat0)\in\overline{\nu_{{\scriptscriptstyle W'}}(C_{{\scriptscriptstyle W'}})}= \overline{\nuW(\CW)}=\supp(\QM)$, see e.g.~the discussion at the end of \cite[Example 7.3]{BaakeGrimm13}.
\end{example}

\subsection{Injectivity properties of $\piG_*$}\label{sec:in}

In order to characterise situations where the factor map $\piG_*:\MW\to\MWG$ is ``nearly'' \oneone, we introduce one more concept that we will use to study the examples in this section. We kept it separate from the main results in Section~\ref{sec:main-results}, because it is not as universal as the results presented there.

\begin{definition}[(Strong) unique coding]\quad\label{def:unique-coding}
\begin{enumerate}[a)]
\item The pair $(W,\LL)$ is \emph{uniquely coding}, 
if for all $h,h'\in H$ holds: $(-h+W)\cap\piH(\LL)=(-h'+W)\cap\piH(\LL)\neq\emptyset$
implies $h=h'$.
\item The pair $(W,\LL)$ is \emph{strongly uniquely coding}, if the following holds:
\begin{quote}
If $h,h',h_n,h_n'\in H$ are such that $h_n\to h$, $h_n'\to h'$ and if
\begin{equation*}
\forall \ell\in\LL\ \exists n_\ell\in\N\ \forall n\geqslant n_\ell: 1_W(h_n+\ell_H)=1_W(h_n'+\ell_H)\ ,
\end{equation*}
where this common value is $1$ for at least one $\tilde{\ell}\in\LL$ for all $n\geqslant n_{\tilde{\ell}}$, then $h=h'$.
\end{quote}
\end{enumerate}
\end{definition}

\begin{lemma}\label{lemma:uniquely-coding}
$(W,\LL)$ is uniquely coding if and only if $\piG_*|_{\nuW(\hX)}$ is \oneone.
\end{lemma}

\begin{lemma}\label{lemma:projection-injective-1}
$(W,\LL)$ is strongly uniquely coding if and only if
$\piG_*:\MW\to \MWG$ is a homeomorphism.
\end{lemma}
\noindent
The slightly technical proofs of these two lemmas are provided at the end of Section~\ref{sec:projected}.

\subsection{Topologically regular windows}\label{subsec:top-reg-wind}

We call a window  $W\subseteq H$ \emph{topologically regular} if $W=\overline{\inn(W)}$, see \cite[Def.~4.11]{Robinson2007}. Dynamical properties of model sets with such windows have been studied e.g.~in \cite{Schlottmann00, FHK2002, Robinson2007, BLM07}.
A window $W\subseteq H$ is \textit{aperiodic} \cite[Def.~5.12]{Robinson2007} or \textit{irredundant} \cite[Def~1]{BLM07}, if   $h+W=W$ implies $h=0$.

\begin{lemma}\label{lemma:regular-window}
Let $W$ be a topologically regular window. Then the following are equivalent.
\begin{itemize}
\item[(i)] $W$ is aperiodic.
\item[(ii)]  $W\ne\emptyset$ and $(W,\LL)$ is strongly uniquely coding.
\item[(iii)]  $W\ne\emptyset$ and $(W,\LL)$ is uniquely coding.
\end{itemize}
\end{lemma}

\begin{proof}

``(i) $\Rightarrow$ (ii)''  $W\ne\emptyset$ is a consequence of aperiodicity.  
Let $h,h',h_n,h_n',n_\ell$ and $\tilde{\ell}$ be as in the definition of strong unique coding (Definition~\ref{def:unique-coding}),
and let $\LL_0=\{\ell\in\LL:\exists n_\ell\in\N\ \forall n\geqslant n_\ell: 1_W(h_n+\ell_H)=1\}$.
Then $\tilde{\ell}\in\LL_0$ by assumption, and 
$h+{\ell}_H,h'+{\ell}_H\in W$ for all $\ell\in\LL_0$, because $W$ is closed.
As $\ell\in\LL_0$ whenever $h+\ell_H\in\inn(W)$, it follows that
\begin{equation*}
W=\overline{\inn(W)}=\overline{\inn(W)\cap(h+\piH(\LL))}
\subseteq\overline{h+\piH(\LL_0)}\ .
\end{equation*}
Let $d:=h'-h$.
Then
$d+h+{\ell}_H=h'+\ell_H\in W$ for each $\ell\in\LL_0$, so that
$d+W\subseteq d+\overline{h+\piH(\LL_0)}=\overline{d+h+\piH(\LL_0)}
=\overline{h'+\piH(\LL_0)}\subseteq \overline{W}=W$, and the aperiodicity of $W$ implies $d=0$, i.e. $h'=h$.
  \\
``(ii) $\Rightarrow$ (iii)'' This is trivial.\\
``(iii) $\Rightarrow$ (i)'' Fix $h'\in H$ such that $(-h'+W)\cap \pi_H(\LL)\ne\emptyset$, which is possible since $W\ne\emptyset$. Consider $h\in H$ such that $h+W=W$. We then have
$(-h'+h+W)\cap \pi_H(\LL)=(-h'+W)\cap \pi_H(\LL)\neq\emptyset$.
As $(W,\LL)$ is uniquely coding, this implies
$-h'+h=-h'$, i.e., $h=0$. Hence $W$ is aperiodic.
\end{proof}

\begin{example}\label{example:projection-ex1}
If $H=\R^d$ and if $W$ is topologically regular, then $\piG_*:\MW\to\MWG$ is a homeomorphism.
This is a special case of the subsequent proposition. \\
\end{example}

\begin{proposition}\label{prop:oneone}
Assume that the only compact subgroup of $H$ is the trivial one.
If $W$ is topologically regular, 
then $(W,\LL)$ is strongly uniquely coding and $\piG_*:\MW\to\MWG$ is a homeomorphism.
\end{proposition}
\begin{proof}
If $W=\emptyset$, the statement is obvious. Hence assume that $W\ne\emptyset$. In view of Lemmas~\ref{lemma:projection-injective-1} and~\ref{lemma:regular-window} 
we only have to show that $W$ is aperiodic. So let $H_W:=\{h\in H: h+W=W\}$.
$H_W$ is a closed subgroup of $H$, and because $H_W+W=W$ is compact and $W$ is non-empty, $H_W$ is compact. Hence $H_W$ is the trivial subgroup.
\end{proof}
\begin{remark}\label{rem:aperiodicity}
This proposition suggests that one can assume irredundancy of the window
without loss of generality by passing from $(G,H,\LL)$ with window $W$ to $(G,H',\LL')$ with window $W'=W/H_W$ where the period group  $H_W=\{h\in H: h+W=W\}$ of $W$ has been factored out. For non-empty topologically regular windows, this is described in detail in \cite[Lemma 7]{BLM07}, where it is shown among others that $W'$ is again topologically regular and that $m_{H/H_W}(\partial W')=0$ if $m_H(\partial W)=0$.
\end{remark}

\begin{example}
A prominent example are the rhombic Penrose tilings, which can be realised as regular model sets in $G=\R^2$ with a topologically regular window in $H=\R^2\times \Z/5\Z$, see \cite[Section~3.2]{Moody97} and \cite[Example~7.11]{BaakeGrimm13}. In fact this description goes back to de Bruijn, see also the discussion in \cite[Section~7.5.2]{BaakeGrimm13}. In that case we have $\hat0\notin \CW$, and the cardinalities $1,2,10$ of the fibres ${\hat\pi}^{-1}\{\hx\}$ for $\hx\in\hX$, compare Definition~\ref{def:hat-pi}, are discussed in \cite{Robinson1996}.
\end{example}

\subsection{$\BB$-free systems}\label{subsec:b-free}

Let $(b_k)_{k\in \N}$ be an increasing sequence of pairwise coprime integers greater than one such that
\begin{equation}\label{eq:B-free-summability}
\sum_{k\in \N} \frac{1}{b_k}<\infty\ .
\end{equation}
Writing  $\BB=\{b_k: k\in\N\} \subset \{2,3,\ldots\}$, the set $V_{\BB}$ of $\BB$-free integers  consists of all integers having no factor in $\BB$, i.e., we have
\begin{displaymath}
V_{\BB}=\Z\setminus \bigcup_{b\in \BB} b\Z\ .
\end{displaymath}
Such sets have been studied in \cite{ALR2013}.  With $\BB=\{p^2: p \text{ prime}\}$, the set $V_{\BB}$ are the square-free integers \cite{Sarnak2011, Peckner2012, CS13}. As remarked in \cite{BaakeHuck14}, the $\BB$-free integers and their lattice generalisations are weak model sets. 
Indeed, consider $G=\Z$, the product group $H=\prod_{k\in \N} \Z/b_k\Z$, and write $h=(h_k)_{k\in\N}$ for $h\in H$. 
The map $\iota:G\to H$, defined by $\iota(g)=(g\mod b_k\Z)_{k\in\N}$, 
is a continuous embedding of $G$ into $H$, and $\iota(G)$ is dense in $H$, because the $b_k$ are pairwise coprime. Define the lattice $\LL\subseteq G\times H$ to be the diagonal embedding $\LL=\{(g,\iota(g)): g\in G\}$. Then $(G,H,\LL)$ is a cut-and-project scheme. The $\BB$-free integers are the weak model set defined by the compact window 
\begin{displaymath}
W=\prod_{k\in \N} \Z/b_k\Z\setminus \{0_k\}\ ,
\end{displaymath} 
where $0$ denotes the neutral element in $H$.

\begin{lemma}\label{lemma:b-free-full}
$\QM$ has full support, i.e., $\supp(\QM)=\MW$.
\end{lemma}

\begin{proof}
We show $\nuW(\hX)\subseteq \supp(\QM)$ using Proposition~\ref{prop:exceptional-points}. Then the claim follows since $\supp(\QM)$ is closed. Fix a fundamental domain $X$ of $\LL$ and take arbitrary $x\in X$. We must show that
\begin{displaymath}
m_H\{h\in H: (-x_H+W)\cap \piH(\LL')=(-h+W)\cap \piH(\LL')\}>0 \qquad \text{for all finite }
\LL'\subseteq \LL\ .
\end{displaymath}
We can of course restrict to sets $\LL'=\{(n, \iota(n)): n\in \{-N,\ldots, N\}\}$ for $N\in \N$. Let $\LL'_x=\{\ell\in \LL':\exists k=k(\ell)\in\N\text{ s.t. }x_{H,k}+\ell_{H,k}=0 \mod b_k\}$. Given $\ell\in\LL'$, we choose the smallest $k$ with this property as $k(\ell)$. Then, for any $\ell\in\LL'$, the condition $\ell_H\in (-x_H+W)\cap \piH(\LL')$ is equivalent to $\ell\in\LL'\setminus\LL'_x$. Let $k_0=\max\{k(\ell):\ell\in \LL'_x\}$ and fix $k_1>k_0$ such that $b_{k_1}>2N$. Observe that $\ell_{H,k}=\iota(n)_k\in \{0,\ldots, N\}\cup \{b_k-N,\ldots, b_k-1\}$ for all $\ell\in \LL'$ and $k>k_1$, and define
\begin{displaymath}
H_x:=\{h\in H: h_k=x_{H,k} \text{ if } k\leqslant k_1,\ h_k\in\{N+1,\ldots, b_k-N-1\} \text{ if }k>k_1\}\ .
\end{displaymath}
Then $m_H(H_x)\geqslant \prod_{k\leqslant k_1}\frac{1}{b_k}\cdot\prod_{k>k_1}\left(1-\frac{2N+1}{b_k}\right)>0$, and it remains to be shown that $(-x_H+W)\cap\piH(\LL')=(-h+W)\cap \piH(\LL')$ for every $h\in H_x$. So let $h\in H_x$ and $\ell\in \LL'$.
If $\ell\in \LL'_x$, then $\ell_H\notin -x_H+W$ and $h_{k(\ell)}+\ell_{H,k(\ell)}=x_{H,k(\ell)}+\ell_{H,k(\ell)}=0 \mod b_{k(\ell)}$, so that $\ell_H\notin -h+W$, too.

If $\ell\in \LL'\setminus \LL'_x$, then $\ell_H\in -x_H+W$, and we show that $\ell_H\in -h+W$ as well: Suppose for a contradiction that this is not the case. Then there is some $k$ such that $h_k+\ell_{H,k}=0 \mod b_k$. But as $x_{H,k}+\ell_{H,k}\ne 0 \mod b_k$ by assumption, we must have $k>k_1$, such that $h_k\in \{N+1,\ldots, b_k-N-1\}$. Together with $\ell_{H,k}\in  \{0,\ldots, N\}\cup \{b_k-N,\ldots, b_k-1\}$, a case analysis implies $h_k+\ell_{H,k}\ne 0 \mod b_k$, a contradiction.
\end{proof}

We study the domain of injectivity of the projection $\piG_*: \MW\to \MWG$.\footnote{Note that $(W,\LL)$ is not uniquely coding: 
Consider for example $h,h'\in H$ given by $h_1=h'_1=1$, $h_2=h'_2=b_2-1$, $h_3=1$,
$h'_3=b_3-1$, $h_{2k}=h_{2k}'=b_{2k}-k$  and
 $h_{2k+1}=h_{2k+1}'=k$ for $k\geqslant2$.
Then $h\neq h'$ and $h,h'\in W$, and it is not hard to check that $h+\iota(k),h'+\iota(k)\not\in W$ for all $k\in\Z\setminus\{0\}$.}
To this end we introduce the following notation:
For $A\subseteq \Z$ and $k\in\N$ we write $\langle A\rangle_k=\{\iota(g)_k: g\in A\}$.
Define $\Y\subseteq \MW$ by
\begin{displaymath}
\Y=\{\nu\in \MW : \card\,\langle\supp(\piG_* \nu)\rangle_k=b_k-1 \text{ for all }k\in\N\}\ .
\end{displaymath}
The set $\Y\subseteq \MW$ is measurable and consists of all measures such that ``every $b_k$-reduction misses exactly one coset''. The set $\piG_*(\Y)$ was studied for square-free integers by Peckner \cite{Peckner2012} (called $X_1$ in his paper), for $\BB$-free systems it is the set $Y$ of \cite{Kulaga-Przymus2014}, and for visible lattice points the set $\A_1$ of \cite{BaakeHuck14}. The next lemmas have close analogues in these three publications.

\begin{lemma}\label{lemma:b-free-1-1}
$\piG_*$ is \oneone on $\Y$ and $\piG_*|_\Y:\Y\to\piG_*(\Y)$ is a Borel isomorphism.
\end{lemma}

\begin{proof}
Choose a fundamental domain $X$ of $\LL$ satisfying $\piG(X)=\{0\}$,
so that each $x\in X$ is of the form $(0,x_H)$. Consider $\nu,\nu'\in \Y$ such that $\piG_*\nu=\piG_*\nu'$. 
Since $\nu,\nu'\ne \underline{0}$, there are unique $x,x'\in X$ such that
\begin{displaymath}
\nu\le \nuW(x+\LL), \qquad \nu'\le\nuW(x'+\LL)\ .
\end{displaymath}
Then $\piG_*\nu\le \piG_*\nuW(x+\LL)$ and $\piG_*\nu'\le\piG_*\nuW(x'+\LL)$. 
Note that $h+W=\prod_k \left(\Z/b_k\Z\setminus \{h_k\}\right)$ for any $h\in H$. We thus have
\begin{equation*}
\langle\supp(\piG_*\nu)\rangle_k
\subseteq
\langle\supp(\piG_*\nuW(x+\LL))\rangle_k
\subseteq
\langle\iota^{-1}(-x_H+W)\rangle_k
=
\Z/b_k\Z\setminus\{-x_{H,k}\}\quad\text{for all }k\in\N\ ,
\end{equation*}
and $\langle\supp(\piG_*\nu')\rangle_k\subseteq\Z/b_k\Z\setminus\{-x_{H,k}'\}$ follows in the same way.
%
%
Since by assumption $ \card\langle \supp(\piG_* \nu)\rangle_k=b_k-1$ 
and $\piG_* \nu = \piG_* \nu'$, we must have $x_{H,k}=x_{H,k}'$ for all $k\in\N$. We thus conclude $x=x'$, which implies $\nu=\nu'$, because $\piG|_\LL$ is \oneone. Hence $\piG_*|_\Y$ is \oneone, and it follows from classical results by Lusin and Souslin \cite[Corollary 15.2]{Kechris1995} that
$\piG_*|_\Y:\Y\to\piG_*(\Y)$ is a Borel isomorphism.
\end{proof}
\begin{lemma}\label{lemma:b-free-Y-full}
We have $\nuW(\hX_{max})\subseteq \Y$.
As a consequence, $\Y$ has full Mirsky measure.
\end{lemma}

\begin{proof}
Take a van Hove sequence $(A_n)_n$ associated to $\hX_{max}$. Then for any $\hat x\in\hX_{max}$ the density of $\nuW(\hat x)$ is given by
\begin{displaymath}
\lim_{n\to\infty} \frac{\nuW(\hat x)(A_n\times H)}{m_G(A_n)}
=
1\cdot m_H(W)
=
\prod_{k\in\N} \left(1-\frac{1}{b_k}\right)\ .
\end{displaymath}
This implies $\card\langle \supp(\piG_* \nuW(\hat x))\rangle_k=b_k-1$ for every $k\in \N$. Indeed, otherwise for some $b_k$ at least two cosets will be missed. But then $\nuW(\hat x)$ has an upper density less than maximal. We thus conclude $\nuW(\hat x)\in \Y$. As $\hat x\in \hX_{max}$ was arbitrary, we have $\nuW(\hX_{max})\subseteq \Y$. Since $\hX_{max}\subseteq \hX$ has full Haar measure, $\Y$ must have full Mirsky measure.
\end{proof}

For $\BB$-free systems, we have the following strengthening of Corollary~\ref{cor:pp-spectrum}.

\begin{proposition}[discrete spectrum for $\BB$-free systems]
For any $\hx\in\hXmax$ we have $\MWG(\hx)=\MWG$, and the measure theoretic dynamical systems
$(\MWG(\hx),\QMG,S)$ and $(\hX,m_\hX,\hT)$ are isomorphic. The configuration $\nuW(\hat 0)$ has maximal density.
\end{proposition}
\begin{proof}
Observe that $\supp(\QM)=\MW$ by Lemma~\ref{lemma:b-free-full}. Hence $\MW(\hx)=\MW$ for all $\hx\in\hXmax$ by Theorem~\ref{theo:Moody-topogical}d, so that
$(\MW(\hx),\QM,S)$ and $(\hX,m_\hX,\hT)$ are isomorphic by Theorem~\ref{theo:projection-mth}a. That these systems are also isomorphic to $(\MWG(\hx),\QMG,S)$ follows now from Lemmas~\ref{lemma:b-free-1-1} and~\ref{lemma:b-free-Y-full}. The density of $\nuW(\hat 0)$ is e.g. computed in \cite{BaakeHuck14}.
\end{proof}

\begin{remark}
The Mirsky measure $\QM$ is not the only invariant probability on $\MW$ that assigns
full mass to the uniqueness set $\Y$. In \cite{Peckner2012} and
\cite{Kulaga-Przymus2014} it is proved that $\piG_*(\Y)$ has full measure also under the measure of maximal entropy. Therefore also this measure can be transferred to $\Y\subseteq\MW$ unambiguously.
\end{remark}

\begin{remark}
Recently, the authors of \cite{BKKL2015} studied the dynamics of more general $\BB$-free systems. They investigate the situation when
the two assumptions that all $b_k$ are pairwise coprime and that $\sum_{k\in\N}1/b_k<\infty$ are weakened or skipped and obtain a wealth of topological, measure theoretic and number theoretic results. With regards to the present setting,
the lack of coprimeness means that $\iota(G)$ is no longer dense in $H$ so that $H$ must be replaced by $\overline{\iota(G)}$. Note, however, that the measure $m_H$ must be replaced by the Haar measure on $\overline{\iota(G)}$. The summability assumption was only used to show that the window has positive Haar measure and that $\supp(\QM)=\MW$ in Lemma~\ref{lemma:b-free-full}.
Our Theorems~\ref{theo:projection-top},~\ref{theo:Moody},~\ref{theo:Moody-extension} and \ref{theo:Moody-topogical} apply without this assumption. So we obtain results similar to some of those in \cite[Theorems A, B and E]{BKKL2015}, and also our definition of the Mirsky measure $\QMG$ corresponds to the one in \cite[Definition 2.30]{BKKL2015}.
\end{remark}

\section{Basic observations}\label{sec:basics}

\begin{lemma}\label{lemma:basic-0}
\begin{enumerate}[a)]
\item $S_g(\nuW(\hx))=\nuW(\hT_g\hx)$ for all $\hx\in\hX$ and $g\in G$.
\item $S_g(\MW)=\MW$ for all $g\in G$.
\end{enumerate}
\end{lemma}
\begin{proof} For each measurable $A\subseteq G\times H$ we have
\begin{displaymath}
\begin{split}
S_g(\nuW(\hx))(A)
&=
\sum_{y\in (x+\LL)\cap(G\times W)}S_g(\delta_y)(A)
=
\sum_{y\in (x+\LL)\cap(G\times W)}\delta_{T_gy}(A)\\
&=
\sum_{y\in T_g((x+\LL)\cap(G\times W))}\delta_y(A)
=
\sum_{y\in (T_gx+\LL)\cap(G\times W)}\delta_y(A)\\
&=
\nuW(\hT_g\hx)(A)\ .
\end{split}
\end{displaymath}
As $S_g^{-1}=S_{-g}$ it suffices to prove that $S_g(\MW)\subseteq \MW$. So let $\nu\in \MW$ be the vague limit of the sequence $(\nuW(\hx^n))_n$. Then $\nuW(\hT_g\hx^n)=S_g\nuW(\hx^n)$ converges vaguely to $S_g\nu$, because $S_g$ is continuous.
\end{proof}

\begin{lemma}\label{lemma:projection-basic2}
Suppose that a sequence $(\nuW(x^n+\LL))_{n\in\N}$, $x^{n}\in G\times H$, converges vaguely to some $\nu\in \MW$.
\begin{enumerate}[a)]
\item If $\lim_{n\to\infty}(x^n+\LL)=\hx=(x+\LL)$, then $\nu=\varphi\cdot\nuW(\hx)$ with 
\begin{equation}\label{eq:project-nu-nux}
\varphi(z)
=
\begin{cases}
0&\text{ if }z\not\in x+\LL\\
0&\text{ if }z\in x+\LL\text{ and }z_H\not\in W\\
\lim_{n\to\infty}1_W((x^{n}+\ell)_H)&\text{ if $z=x+\ell$ for some $\ell\in \LL$ and }z_H\in\partial W \\1&\text{ if }z\in x+\LL\text{ and }z_H\in\inn(W)
\end{cases}
\end{equation}
\item If $\lim_{n\to\infty}(x^n+\LL)=\hx=(x+\LL)$ and $\nuW(\hx)(G\times \partial W)=0$, then $\nu=\nuW(\hx)$.
\item If $\nu$ is not the zero measure, then there exists $\hx\in \hX$ such that $\lim_{n\to\infty}(x^n+\LL)=\hx$.
\end{enumerate} 
\end{lemma}
\begin{proof}
a) As $\lim_{n\to\infty}(x-x^n+\LL)=\LL$ (in $\hX$), there are $\ell^n\in\LL$ such that $\lim_{n\to\infty}(x-x^n+\ell^n)=0$ (in $G\times H$). Replacing all $x^n$ by $x^n+\ell^n$ in the lemma does not change its assumptions nor its assertions, because they are all formulated in terms of the $x^n+\LL$. Therefore we may assume that all $\ell^n=0$ and hence $\lim_{n\to\infty}x^n=x$.

 Let ${V}\subseteq G\times H$ be an open zero neighbourhood whose closure is compact and for which all sets $\overline {V}+\ell$ $(\ell\in \LL)$ are pairwise disjoint.
For each $z\in G\times H$ we fix some open neighbourhood $V_z\subseteq {V}+z$ of $z$ in the following way:
\begin{enumerate}[-]
\item If $z\not\in x+\LL$, then $\overline{V_z}\cap(x+\LL)=\emptyset$.
\item If $z\in x+\LL$ and $z_H\not\in W$, then $\overline{V_z}\cap(G\times W)=\emptyset$.
\item If $z\in x+\LL$ and $z_H\in\inn(W)$, then 
$\overline{V_z}\subseteq (G\times\inn(W))$.
\end{enumerate}
Finally we fix compact $z$-neighbourhoods $C_z\subseteq V_z$, which is always possible in a locally compact space.
Vague convergence implies that
\begin{displaymath}
\nu(V_z)
\leqslant
\liminf_{n\to\infty}\nuW(x^{n}+\LL)(V_z)
\quad\text{and}\quad
\nu(C_z)
\geqslant
\limsup_{n\to\infty}\nuW(x^{n}+\LL)(C_z)\ .
\end{displaymath}
It follows that 
\begin{displaymath}
\begin{split}
\nu(V_z)
&\leqslant
\liminf_{n\to\infty}
\sum_{y\in(x^n+\LL)\cap (G\times W)}\delta_y(V_z)\\
&=
\begin{cases}
0&\text{ if }z\not\in x+\LL\\
0&\text{ if }z\in x+\LL\text{ and }z_H\not\in W\\
1&\text{ if }z\in x+\LL\text{ and }z_H\in\inn(W)\\
\liminf_{n\to\infty}1_W((x^{n}+\ell)_H)&\text{ if $z=x+\ell$ for some $\ell\in \LL$ and }z_H\in\partial W 
\end{cases}\\
&\leqslant
1_{\{x+\LL\}}(z)\cdot 1_{G\times W}(z)\\
&=
\nuW(x+\LL)(V_z)
\end{split}
\end{displaymath}
and
\begin{displaymath}
\begin{split}
\nu(C_z)
&\geqslant
\limsup_{n\to\infty}\nuW(x^{n}+\LL)(C_z)
=
\limsup_{n\to\infty}
\sum_{y\in(x^n+\LL)\cap (G\times W)}\delta_y(C_z)\\
&=
\begin{cases}
0&\text{ if }z\not\in x+\LL\\
0&\text{ if }z\in x+\LL\text{ and }z_H\not\in W\\
1&\text{ if }z\in x+\LL\text{ and }z_H\in\inn(W)\\
\limsup_{n\to\infty}1_W((x^{n}+\ell)_H)&\text{ if $z=x+\ell$ for some $\ell\in \LL$ and }z_H\in\partial W 
\end{cases}
\end{split}
\end{displaymath}
As $\nuW(x+\LL)$ is a sum of isolated unit point masses in the set $x+\LL$, the combination of both inequalities implies $\nu=\varphi\cdot\nuW(x+\LL)$ with $\varphi$ from (\ref{eq:project-nu-nux}).
\\
b) Consider any $z=x+\ell$, $\ell\in \LL$, with $\nuW(\hx)(\{z\})=1$. We only need to show that $\varphi(z)=1$. But
$\nuW(\hx)(\{z\})=1$ implies by definition of $\nuW$ that $z_H\in W$. From $\nuW(\hx)(G\times \partial W)=0$ we conclude that 
$z_H\not\in\partial W$. Hence $z_H\in\inn(W)$, and $\varphi(z)=1$ in view of (\ref{eq:project-nu-nux}).\\
c) As $\hat X$ is compact, the sequence $(x^n+\LL)_n$ has a subsequence that converges to some $\hx\in\hX$. Applying assertion a) to such a subsequence we conclude that $\nu\leqslant\nuW(\hx)$.
As $\nu$ is different from the zero measure, this determines $\hx$ uniquely and independently of the initially chosen subsequence.
\end{proof}

\begin{lemma}\label{lemma:proj-homeo-1}
$\pihXG_*:\GMW\to\GMWG$ is a homeomorphism (that respects the fibres $(\pihX_*)^{-1}\{\hx\}$).
\end{lemma}
\begin{proof}
As $\pihXG_*$ is a continuous surjective map between the compact space $\GMW$ and the Hausdorff space $\GMWG$ (Lemma~\ref{lemma:diagrams-basic}), only its injectivity needs to be shown.
So let $(\hx_1,\nu_1), (\hx_2,\nu_2)\in\GMW$ and assume that 
$\pihXG_*(\hx_1,\nu_1)=\pihXG_*(\hx_2,\nu_2)$, i.e., $(\hx_1,\piG_*\nu_1)=(\hx_2,\piG_*\nu_2)$, in particular $\hx:=\hx_1=\hx_2=x+\LL$. By Lemma~\ref{lemma:projection-basic2}a), $\nu_i=\varphi_i\cdot\nuW(\hx)$ $(i=1,2)$, so that
there are $A_1,A_2\subseteq\LL$ such that $\nu_i=\sum_{\ell\in A_i}\delta_{x+\ell}$ and hence
$\piG_*\nu_i=\sum_{\ell\in A_i}\delta_{\piG x+\piG\ell}=\sum_{g\in\piG(A_i)}\delta_{\piG x+g}$.
As $\piG_*\nu_1=\piG_*\nu_2$, it follows that $\piG(A_1)=\piG(A_2)$, and as $\piG|_\LL$ is injective, we conclude that $A_1=A_2$ and hence $\nu_1=\nu_2$.
\end{proof}

\begin{lemma}\label{lemma:projection-unique-nu-1}
For each $\nu\in \MW\setminus\{\0\}$ there is a unique $\hx\in\hX$ such that  $\nu\leqslant\nuW(\hx)$.
\end{lemma}
\begin{proof}
Each $\nu\in \MW$ is the limit of a sequence of measures $\nuW(\hx^n)$. If $\nu\neq\0$, there is $\hx\in\hX$ such that $\lim_{n\to\infty}\hx^n=\hx$ and $\nu\leqslant\nuW(\hx)$ by Lemma~\ref{lemma:projection-basic2}. This $\hx$ is obviously unique.
\end{proof}

\begin{definition}\label{def:hat-pi}
Given $\nu\in \MW\setminus\{\0\}$, we denote the unique $\hx\in\hX$ with $\nu\leqslant\nuW(\hx)$ by $\hat\pi(\nu)$. This defines a map $\hat\pi:\MW\setminus\{\0\}\to\hX\setminus \ZW$.
\end{definition}

\begin{lemma}\label{lemma:projection-hatpi}
$\hat\pi$ is continuous on $\MW\setminus\{\0\}$ and $\hat\pi\circ S_g=\hT_g\circ\hat\pi$ for all $g\in G$.
\end{lemma}
\begin{proof}
Suppose for a contradiction that $\nu,\nu^n$ are in $\MW\setminus\{\0\}$ and $\nu^n\to\nu$ vaguely, but there exists $U\subseteq G\times H$ such that its image $\hat U\subseteq \hX$ under the quotient map is an open neighbourhood of $\hat\pi(\nu)$ satisfying $\hat\pi(\nu^n)\not\in{\hat U}$ for all $n$. Then $\nu^n( U)=0$ for all $n$ but $\nu(U)>0$, which contradicts the vague convergence. Hence $\hat\pi$ is continuous.
Now let $\nu\in \MW\setminus\{\0\}$ and $g\in G$. As $\nu\leqslant\nuW(\hat\pi(\nu))$, we have
$S_g\nu\leqslant S_g\nuW(\hat\pi(\nu))=\nuW(\hT_g\hat\pi(\nu))$. Hence $\hat\pi(S_g\nu)=\hT_g\hat\pi(\nu)$.
\end{proof}

\section{Zeros and continuity: Proof of Proposition~\ref{prop:projection-usc}}

Some of the observations of this section appear in various disguises in the literature, see \cite[Remark 7.4]{BaakeGrimm13} and the references given there. For example, the Baire argument as in the proof of Proposition~\ref{prop:projection-usc} was used e.g.~in \cite{Schlottmann1993, bms98, Schlottmann00, BaakeMoody2004, BLM07, HuckRichard15}. In the literature on invariant graphs for skew product transformations, the Baire argument is traditionally used to prove that the set of zeros or the set of continuity points of the graph is a dense $G_\delta$, see e.g.~\cite{Keller1996,Stark2003,Jager2003}. Implicitly this observation is already contained in \cite{Herman1983}. 

\begin{lemma}\label{lemma:projection-CW}
Let $x\in G\times H$. Then $x_H\in\bigcap_{\ell\in \LL}((\partial W)^c-\ell_H)$ if and only if $(x+\LL)\in \CW$. Equivalently:
\begin{equation*}
\CW=\pihX\bigg((\piH)^{-1}\bigg(\bigcap_{\ell\in \LL}((\partial W)^c-\ell_H)\bigg)\bigg)
\end{equation*}
\end{lemma}
\begin{proof}
Let $x+\LL,x^n+\LL\in\hX$, $\lim_{n\to\infty}(x^n+\LL)=(x+\LL)$. Assume that $x_H\in\bigcap_{\ell\in \LL}((\partial W)^c-\ell_H)$, i.e., $(x+\ell)_H\not\in\partial W$ for all $\ell\in \LL$. Then $\nuW(x+\LL)(G\times \partial W)=0$, and Lemma~\ref{lemma:projection-basic2} implies that $\lim_{n\to\infty}\nuW(x^n+\LL)=\nuW(x+\LL)$. Hence $x+\LL\in \CW$.

Conversely, if $x_H\not\in\bigcap_{\ell\in \LL}((\partial W)^c-\ell_H)$, then there is $\ell\in \LL$ such that $x_H+\ell_H\in\partial W$. Hence there are $x^n_H\in W^c-\ell_H$ such that $x_H^n\to x_H$ as $n\to\infty$. Let $x^n=(x_G,x_H^n)$. Then, for each sufficiently small open neighbourhood $U$ of $x+\ell$ in $G\times H$ and sufficiently large $n$ we have $\nuW(x^n+\LL)(U)=1_W(x_H^n+\ell_H)\cdot\delta_{x_n+\ell}(U)=0$ while 
$\nuW(x+\LL)(U)=1_W(x_H+\ell_H)\cdot\delta_{x+\ell}(U)=1$ so that $\nuW(x^n+\LL)\not\to\nuW(x+\LL)$. In particular, $(x+\LL)\not\in \CW$.
\end{proof}

\begin{lemma}\label{lemma:int(CW)}
$\inn(\CW)=\emptyset$ if and only if $\partial W\neq\emptyset$.
\end{lemma}
\begin{proof}
If $\partial W=\emptyset$, then $\CW=\hX$ by Lemma~\ref{lemma:projection-CW}.
Conversely, let $\partial W\neq\emptyset$ and
suppose for a contradiction that $\inn(\CW)\neq\emptyset$. Then $(\pihX)^{-1}( \CW)\subseteq G\times H$ has non-empty interior, so that $\piH((\pihX)^{-1}(\CW))$ contains a non-empty open set $U$, as $\piH$ is open.
Because of Lemma~\ref{lemma:projection-CW}, $U\subseteq\bigcap_{\ell\in \LL}((\partial W)^c-\ell_H)$. This implies that $\bigcup_{\ell\in\LL}(U+\ell_H)\subseteq(\partial W)^c$.
But as $\{\ell_H:\ell\in\LL\}=\piH(\LL)$ is dense in $H$ and as $\partial W\neq\emptyset$,
this is impossible.
\end{proof}

\begin{lemma}\label{lemma:projection-ZC}
If $A\subseteq \GMW$ is a non-empty closed $S$-invariant subset, then $A\supseteq\overline{\cG({\nuW}|_{\CW})}$, the $S$-invariant closure of the graph of the restriction of $\nuW$ to the set $\CW$.
\end{lemma}
\begin{proof}
The $S$-invariance of $\overline{\cG({\nuW}|_{\CW})}$ holds since $\CW$ is $\hT$-invariant, compare the proof of Lemma~\ref{lemma:basic-0}b.  As $A\neq\emptyset$ is $S$-invariant, the set $\pihX_*(A)\neq\emptyset$ is $\hT$-invariant, and as $\hT$ is minimal, $\pihX_*(A)=\hX\supseteq \CW$.
Because of Remark~\ref{rem:CW}, $(\pihX_*)^{-1}\{\hx\}=\{(\hx,\nuW(\hx))\}$ for each $\hx\in \CW$.
Hence
$\cG({\nuW}|_{\CW})\subseteq A$, and as $A$ is closed, the lemma is proved.
\end{proof}

\paragraph{Proof of Proposition~\ref{prop:projection-usc}}\quad\\
a) This assertion is contained in Lemma~\ref{lemma:projection-basic2}.\\
b) Because of a), $\hx\in \CW$ if and only if $\{\nu\in \MW: \nu\leqslant\nuW(\hx)\}=\{\nuW(\hx)\}$.\\
c) We show that $\CW$ is a dense $G_\delta$-set: As $\inn(\partial W)=\emptyset$, the closed set $\partial W$ is nowhere dense and so are all translates $(\partial W)-\ell_H$ for $\ell\in \LL$. As $H$ is a Baire space, the set $\bigcap_{\ell\in \LL}((\partial W)^c-\ell_H)$ is a dense $G_\delta$-set in $H$.
As $\piH$ is continuous and open, $(\piH)^{-1}\left(\bigcap_{\ell\in \LL}((\partial W)^c-\ell_H)\right)$ is a dense $G_\delta$-set in $G\times H$. As the quotient map onto $\hX$ is continuous open, we conclude with  Lemma~\ref{lemma:projection-CW} that $\CW$ is a dense $G_\delta$-set in $\hX$.
\footnote{The sets $U_\ell:=(\piH)^{-1}((\partial W)^c-\ell_H)$ are open and dense in $G\times H$. As $\bigcap_\ell U_\ell$ is invariant under translations by $\ell\in\LL$
 and as $\pihX|_{X}$ is bijective, we have $\pihX(\bigcap_\ell U_\ell)=\pihX(X\cap\bigcap_\ell U_\ell)=\bigcap_\ell\pihX(X\cap U_\ell)=\bigcap_\ell\pihX(U_\ell)$. As $\pihX$ is open, this set is $G_\delta$, and as $\pihX$ is continuous and onto, it is also dense.
}
\\
d) Assume now that $\inn(W)=\emptyset$. Then $W^c=(\partial W)^c$. By definition, $\hx=x+\LL\in \ZW$ if and only if $(x+\LL)\cap(G\times W)=\emptyset$, so that $\hx=x+\LL\in \ZW$ if and only if $x_H\in\bigcap_{\ell\in \LL}(W^c-\ell_H)=\bigcap_{\ell\in \LL}((\partial W)^c-\ell_H)$. Hence $\ZW=\CW$ by Lemma~\ref{lemma:projection-CW}. 
It remains to show that $\0\in\overline{\nuW(\hX\setminus Z_W)}$ if $W\ne\emptyset$. Pick some $\hx\in \ZW$, which is possible since $\CW=\ZW$ is dense in the nonempty set $\hX$.
As $\inn(W)=\emptyset$, we have $\partial W=W\neq\emptyset$. Therefore, Lemma~\ref{lemma:int(CW)} shows that $\inn(Z_W)=\inn(\CW)=\emptyset$. Hence we find $\hx_n\in\hX\setminus Z_W$ with $\hx_n\to\hx$.
It follows that $\0=\nuW(\hx)=\lim_{n\to\infty}\nuW(\hx_n)$, i.e., $\0\in\overline{\nuW(\hX\setminus Z_W)}$.
 \\
e) Let $\inn(W)\neq\emptyset$ and assume for a contradiction that $\0\in \MW$. Then there are $x^n\in X$ such that $\lim_{n\to\infty}\nuW(x^n+\LL)=\0$. As $\piH(\LL)$ is dense in $H$ and as $\piH(X)$ is relatively compact in $H$, there is a finite set $\LL_0\subseteq \LL$ such that
$\piH(X)\subseteq\bigcup_{\ell\in \LL_0}(W-\ell_H)$. Let $Q$ be a compact subset of $G\times H$ that contains all sets $X+\ell$, $\ell\in \LL_0$. As $Q$ is compact,
$\limsup_{n\to\infty}\nuW(x^n+\LL)(Q)\leqslant0$,
so that there is $n_0\in\N$ such that $\nuW(x^n+\LL)(Q)=0$ for all $n\geqslant n_0$.
As $x_n+\ell\in X+\ell\subseteq Q$ for all $n$ and all $\ell\in \LL_0$, this implies that $(x^n+\ell)_H\not\in W$ for $n\geqslant n_0$ and $\ell\in \LL_0$, i.e., $\piH(x^n)\not\in\bigcup_{\ell\in \LL_0}(W-\ell_H)$ for all $n\geqslant n_0$. As $x^n\in X$, this contradicts the above choice of $\LL_0$.

\section{The projected system}\label{sec:projected}

\begin{lemma}\label{lemma:projection-piG1}
\begin{enumerate}[a)]
\item All measures in $\MWG$ are uniformly locally finite. 
\item $\piG_*:\MW\to \MWG$ is continuous.
\item $\MWG$ is vaguely compact.
\item For any $x,x'\in G\times H$, either $\piG_*\nuW(x+\LL)\perp\piG_*\nuW(x'+\LL)$ or $\piG(x'-x)\in\piG(\LL)$.
\end{enumerate}
\end{lemma}
\begin{proof}
a)\quad
This follows from the same property of $\MW$ together with the compactness of $W\subseteq H$.\\
b)\quad Let $f\in C_c(G)$. Then $f\circ \piG$ is continuous but typically does not have compact support. However, as $H$ is locally compact, there is a function $g_{\scriptscriptstyle W}\in C_c(H)$ such that $1_W\leqslant g_{\scriptscriptstyle W}$. Hence we have for any $\nu^n,\nu\in \MW$: $\int f\circ\piG\,d\nu^n\to\int f\circ\piG\,d\nu$ if and only if $\int f\circ\piG\cdot g_{\scriptscriptstyle W}\circ\piH\,d\nu^n\to\int f\circ\piG\cdot g_{\scriptscriptstyle W}\circ\piH\,d\nu^n$. 
As $f\circ\piG\cdot g_{\scriptscriptstyle W}\circ\piH$ has compact support, we can conclude that $\piG_*\nu^n$ converges vaguely to $\piG_*\nu$ whenever $\nu^n$ converges vaguely to $\nu$. This proves that $\piG_*$ is continuous.\\
c)\quad
Because of b), $\MWG$ is the continuous image of a compact set. Hence it is compact.\\
d) Let $x,x'\in X$ and suppose that $\piG_*\nuW(x+\LL)\not\perp\piG_*\nuW(x'+\LL)$.
Then there is $g\in G$ such that $\piG_*\nuW(x+\LL)(\{g\})>0$ and $\piG_*\nuW(x'+\LL)(\{g\})>0$. It follows that there are $\ell,\ell'\in \LL$ such that $\piG(x+\ell)=g=\piG(x'+\ell')$ and hence
$\piG(x'-x)=\piG(\ell-\ell')\in\piG(\LL)$.
\end{proof}

\paragraph{Proof of Proposition~\ref{prop:projection-homeo}}\quad\\
a) This is Lemma~\ref{lemma:projection-piG1}b.\\
b) $\piGH_*:\GMW\to\MW$ is \oneone on $(\piGH_*)^{-1}(\MW\setminus\{\0\})$ by Lemma~\ref{lemma:projection-unique-nu-1}. Suppose now that $\inn(W)\neq\emptyset$. Then $\0\not\in \MW$ because of  Proposition~\ref{prop:projection-usc}e, so that $\piGH_*$ is \oneone. As it is
a continuous and surjective map between the compact space $\GMW$ and the Hausdorff space $\hX$, it is a homeomorphism.\\
c) By Proposition~\ref{prop:projection-usc}, $\CW$ is a dense $G_\delta$-subset of $\hX$ and 
$(\pihX_*)^{-1}\{\hx\}=\{(\hx, \nuW(\hx))\}$ for every $\hx\in \CW$. Hence $\pihX_*:\GMW\to\hX$ is an almost \oneone extension of an equicontinuous factor. This extension can be factorised over the maximal equicontinuous factor of $(\GMW,S)$, call it $(Y,R)$. Then also $(Y,R)$ is an almost \oneone extension of $(\hX,\hT)$, and Lemma~\ref{lemma:projection-top-dyn1} below
shows that it coincides with $(\hX,\hT)$.\qed
\\[5mm]
The following lemma is well known when $(Y,R)$ is minimal, but we could not locate this slight generalisation in the literature. Recall that a metric $G$-dynamical system $(Y,R)$ is distal if for any pair $x\ne y$ in $Y$ there is $\epsilon>0$ such that $d(R_gx,R_gy)\ge\epsilon$ for all $g\in G$. For example, if $(Y,R)$ is equicontinuous, then $(Y,R)$ is distal. 

\begin{lemma}\label{lemma:projection-top-dyn1}
Let $\pi:(Y,R)\to(X,S)$ be an extension of compact metrisable $G$-dynamical systems and suppose that
\begin{compactenum}
\item $(Y,R)$ is distal,
\item $(X,S)$ is minimal, and
\item there exist some $x_0\in X$ such that $\card(\pi^{-1}\{x_0\})=1$.
\end{compactenum}
Then $\pi$ is a homeomorphism.
\end{lemma}
\begin{proof}
Let $Z$ be a minimal subset of $Y$. Then $\pi(Z)=X$, as $(X,S)$ is minimal. Hence $\pi^{-1}\{x_0\}\subseteq Z$. In particular, $Y$ has only one minimal subset. But as $(Y,R)$ is distal, it is the disjoint union of its minimal subsystems \cite[Prop.~II.7]{Furstenberg1967}. Hence $(Y,R)$ is minimal. The rest of the proof is classical:

Let $y_0$ be the unique point in $\pi^{-1}\{x_0\}$. Suppose  $y_1, y_2\in Y$ are such that $\pi(y_1)= \pi(y_2)$. Then there is a sequence $(g^n)_n$ in $G$ such that
$\lim_{n\to\infty}R_{g^n}y_1=y_0$, and passing to a subsequence if necessary, we may assume that $y_0':=\lim_{n\to\infty}R_{g^n}y_2$ exists as well. Then $\pi(y_0')=\lim_{n\to\infty}\pi(R_{g^n}y_2)=\lim_{n\to\infty}S_{g^n}(\pi(y_2))=\lim_{n\to\infty}S_{g^n}(\pi(y_1))=\lim_{n\to\infty}\pi(R_{g^n}y_1)=\pi(y_0)=x_0$ so that $y_0'=y_0$.
As $(Y,R)$ is distal, this implies $y_1=y_2$. Since $\pi$ is a \oneone and surjective map between a compact and a Hausdorff space, it follows that $\pi$ is a homeomorphism.
\end{proof}
\paragraph{Proof of Theorem~\ref{theo:projection-top}}\quad\\
a) The first claim is a corollary to Proposition~\ref{prop:projection-homeo}\, b and c. We turn to the subsystems $\overline{\cG({\nuW}|_{\CW})}\subseteq \GMW$ and $\overline{\nuW(\CW)}\subseteq\MW$. As the second subsystem is a continuous factor of the first one, and as both are topological almost \oneone extensions of $(\hX,\hT)$ by the first claim, it suffices to note that the restriction of $S$ to ${\overline{\cG({\nuW}|_{\CW})}}$ is 
the unique minimal subsystem of $(\GMW,S)$ in view of Lemma~\ref{lemma:projection-ZC}.
\\
b) 
The first claim is Proposition~\ref{prop:projection-homeo}\, c.
As $\ZW\neq\emptyset$, we have $\0\in \MW$, and of course $\0$ is a fixed point for the action of $S$ on $\MW$. Therefore, any factor of $(\MW,S)$ has a fixed point, so any minimal factor must be trivial. As $(\MWG,S)$ is a factor of $(\MW,S)$, the same holds for $(\MWG,S)$.

We only prove the failure of weak mixing for the system $(\MWG,S)$ and observe that it is a factor of $(\MW,S)$. To this end assume that $(\MWG,S)$ is topologically weak mixing, and 
observe that $\GMWG\subseteq\hX\times \MWG$ is a
topological joining between $(\hX,\hT)$ and $(\MWG,S)$. 
As $(\hX,\hT)$ is minimal and distal (since it is even equicontinuous), an old disjointness result of Furstenberg \cite[Theorem II.3]{Furstenberg1967} then shows that the joining is trivial, i.e., that $\GMWG=\hX\times \MWG$; see \cite{BaakeHuck14} for this argument. 
As $\CW\neq\emptyset$, 
Proposition~\ref{prop:projection-usc}c) and  Remark~\ref{rem:CW} show that 
$\card(\MWG)=1$.

\paragraph{Proof of Corollary~\ref{cor:projection-top}b}\quad\\
Suppose that $\inn(W)\neq\emptyset$ and let $W':=\overline{\inn(W)}$. Then $\inn(W')=\inn(W)$ so that $W'$ is topologically regular (see Subsection~\ref{subsec:top-reg-wind}) and $\partial W'\subseteq\partial W$. Therefore $\CW\subseteq\CWprime$ and $\nuWG|_{\CW}=\nuWGprime|_{\CW}$. Hence $\overline{\nuWG(\CW)}=\overline{\nuWGprime(\CW)}\subseteq\overline{\nuWGprime(\CWprime)}$, and as $\overline{\nuWGprime(\CWprime)}\subseteq\overline{\nuWGprime(\CW)}$ by Remark~\ref{remark:support}a, we conclude that \linebreak
$M:=\overline{\nuWGprime(\CWprime)}=\overline{\nuWG(\CW)}\subseteq\MWG\cap\MWGprime$. As $W'$
is assumed to be aperiodic, Lemmas~\ref{lemma:projection-injective-1} and~\ref{lemma:regular-window} show that $\piG_*:\MWprime\to\MWGprime$ is a homeomorphism. Hence, by Corollary~\ref{cor:projection-top},
$(\overline{\nuWG(\CW)},S)=(\overline{\nuWGprime(\CWprime)},S)$ is an almost automorphic extension of $(\hX,\hT)$, and as $(\MWG,S)$ is a topological factor of $(\MW,S)$, 
$\overline{\nuWG(\CW)}$ is the only minimal subsystem of $(\MWG,S)$, see also Theorem~\ref{theo:projection-top}a.

\paragraph{Proof of Proposition~\ref{prop:projection-mth}}\quad\\
a) The $\hT$-invariance is Remark \ref{rem:CW}, the remaining statement is obvious from b) and c) below.

\noindent b) Suppose first that $m_H(W)>0$.
As the action $h\mapsto \ell_H+h$ of $\LL$ on $(H,m_H)$ is metrically transitive, it follows easily that the set $\bigcup_{\ell\in \LL}(W-\ell_H)$ has conull $m_H$-measure, so that
\begin{displaymath}
m_{G\times H}\left((\piH)^{-1}\left(\bigcap_{\ell\in \LL}(W^c-\ell_H)\right)\right)
=
m_{G\times H}\left(G\times \left(\bigcup_{\ell\in \LL}(W-\ell_H)\right)^c\right)
=0\ .
\end{displaymath}
Hence
\begin{displaymath}
m_\hX(\ZW)
=
\dL\cdot m_{G\times H}\left(X\cap(\piH)^{-1}\left(\bigcap_{\ell\in \LL}(W^c-\ell_H)\right)\right)
=
0\ .
\end{displaymath}
Conversely, if $m_H(W)=0$, then
\begin{displaymath}
\begin{split}
m_\hX(\hX\setminus\ZW)\leqslant \dL\cdot m_{G\times H}\left(G\times \left(\bigcup_{\ell\in\LL}(W-\ell_H)\right)\right)=0\ .
\end{split}
\end{displaymath}

\noindent c) Recall that $\hx=x+\LL\in \CW$ if and only if $x_H\in\bigcap_{\ell\in \LL}((\partial W)^c-\ell_H)$ by Lemma~\ref{lemma:projection-CW}. Hence the claim follows by the same arguments as in b), with $W$ replaced by $\partial W$.
\qed

\paragraph{Proof of Theorem~\ref{theo:projection-mth}}\quad\\
a) As $m_H(W)>0$, we have $m_\hX(Z_W)=0$ because of Proposition~\ref{prop:projection-mth} b). This implies $\QGM(\hX\times\{\0\})=0$, so that 
$(\GMW,\QGM,S)$ and $(\MW,\QM,S)$ are measure theoretically isomorphic in view of Proposition~\ref{prop:projection-homeo} b).
 Since by definition $\QGM=m_\hX\circ(\id,\nuW)^{-1}$, it follows that 
$(\id,\nuW):\hX\to\GMW$ is a measure theoretic isomorphism between 
 the systems $(\hX,m_\hX,\hT)$ and $(\GMW,\QGM,S)$ 
with inverse $\pihX_*$.
\\
b) $(\hX\times\MWG,\QGM,\hT\times S)$ is obviously a joining of $(\hX,m_\hX,\hT)$ and  $(\MWG,\QMG,S)$. If it is trivial, then $\QGM$ is a product measure that is supported by the graph of $\nuWG$ so that $\nuWG$ is constant $m_\hX$-a.e.
\\
c) 
By assumption we must have $\inn(W)\ne\emptyset$, which implies $\ZW=\emptyset$ by Proposition~\ref{prop:projection-usc}. Since $m_H(\partial W)=0$, we have $m_\hX(\CW)=1$, that is $\nuW:\hX\to\MW$ is $m_\hX$-a.e. continuous. It follows that $\{\hx\in\hX: \card((\pihX_*)^{-1}\{\hx\})=1\}$ has full Haar measure. In particular, each invariant measure on $\GMW$ that projects to $m_\hX$ coincides with $\QGM$, and as each such invariant measure projects to the unique invariant measure $m_\hX$ of $(\hX,\hT)$, the system $(\GMW,S)$ is indeed uniquely ergodic. As $(\MW,S)$ and $(\MWG,S)$ are continuous factors of $(\GMW,S)$, the unique ergodicity immediately carries over to these systems.
\qed

\paragraph{Proof of Lemma~\ref{lemma:uniquely-coding}}\quad\\
``$\Rightarrow$'' 
Assume that $\hx=x+\LL,\hx'=x'+\LL\in\hX$ are such that $\piG_*\nuW(\hx)=\piG_*\nuW(\hx')$. 

If one of $\nuW(\hx)$ and $\nuW(\hx')$ is $\0$, then also 
$\piG_*\nuW(\hx)=\piG_*\nuW(\hx')=\0$ and hence $\nuW(\hx)=\0=\nuW(\hx')$.

So assume from now on that $\nuW(\hx),\nuW(\hx')\neq\0$.
By Lemma~\ref{lemma:projection-piG1}d, 
there is some $\tilde{\ell}\in \LL$ such that $x'_G=x_G+\tilde\ell_G$.
Let $\tilde{x}=x+\tilde\ell$. Then $\tilde{x}_G=x'_G$, $\nuW(x+\LL)=\nuW(\tilde{x}+\LL)$, and
\begin{displaymath}
\sum_{\ell\in \LL}1_W(\tilde x_H+\ell_H)\,\delta_{\tilde x_G+\ell_G}
=
\piG_*\nuW(\tilde x+\LL)
=
\piG_*\nuW(x+\LL)
=
\piG_*\nuW(x'+\LL)
=
\sum_{\ell\in \LL}1_W(x'_H+\ell_H)\,\delta_{\tilde x_G+\ell_G}\ .
\end{displaymath}
Therefore, for each
 $\ell\in \LL$, $(\tilde x_H+\ell_H)\in W$ if and only if $(x_{H}'+\ell_H)\in W$, which can be expressed as $(-\tilde{x}_H+W)\cap\piH(\LL)=(-x'_H+W)\cap\piH(\LL)$.
These sets  are non-empty, because  $\piG_*\nuW(x'+\LL)\neq\0$.
As $(W,\LL)$ is uniquely coding this implies $\tilde x_H=x_H'$ and hence $\tilde x=x'$, so that $x'-x=\tilde\ell\in \LL$, i.e., $\hx=\hx'$ and hence $\nuW(\hx)=\nuW(\hx')$.

``$\Leftarrow$'' Let $h,h'\in H$ be given such that $(-h+W)\cap \piH(\LL)=(-h'+W)\cap \piH(\LL)\ne\emptyset$. Then we have  $\piG_*\nuW((0,h)+\LL)=\piG_*\nuW((0,h')+\LL)\ne\0$, so that by assumption also $\nuW((0,h)+\LL)=\nuW((0,h')+\LL)\neq\0$. 
Then Lemma~\ref{lemma:projection-unique-nu-1} implies that $(0,h)+\LL=(0,h')+\LL$ so that
$(0,h-h')\in \LL$. As $\piG|_\LL$ is \oneone, this means $h=h'$. Hence $(W,\LL)$ is uniquely coding.
\qed

\paragraph{Proof of Lemma~\ref{lemma:projection-injective-1}}\quad\\
Suppose first that $(W,\LL)$ is strongly uniquely coding.
Because of Lemma~\ref{lemma:projection-piG1}b and compactness of $\MW$ it only remains to prove that $\piG_*$ is \oneone. 
Suppose that $\piG_*\nu=\piG_*\nu'$ for some $\nu,\nu'\in \MW$. If 
$\piG_*\nu=\piG_*\nu'=\0$,
then also $\nu=\0=\nu'$. Otherwise Lemma~\ref{lemma:projection-unique-nu-1} shows that there are unique $x,x'\in X$ such that $\nu\leqslant\nuW(x+\LL)$ and $\nu'\leqslant\nuW(x'+\LL)$. Hence, $\piG_*\nu\leqslant\piG_*\nuW(x+\LL)$ and
$\piG_*\nu'\leqslant\piG_*\nuW(x'+\LL)$, and as $0\neq\piG_*\nu=\piG_*\nu'$,
Lemma~\ref{lemma:projection-piG1}d implies that 
$x_G'=x_G+\tilde{\ell}_G$ for some $\tilde{\ell}\in\LL$, and as in the previous proof we can replace $x$ by $\tilde{x}=x+\tilde{\ell}$ such that $\nuW(\tilde{x}+\LL)=\nuW(x+\LL)$ and $\tilde{x}_G=x'_G$. Hence we can assume without loss that $x_G=x_G'$, and it remains to show that $x_H=x_H'$.

Let $G\times H=\bigcup_{j\in\N}K_j$, where $(K_j)_j$ is an increasing sequence of compact subsets of $G\times H$. Set $\LL_j:=\LL\cap K_j$ so that
$\LL=\bigcup_{j\in\N}\LL_j$ and all $\LL_j$ are finite. As $\nu,\nu'$ are in the closure of $\nuW(\hX)$, there are 
$\hx_n=(x_{n,G},x_{n,H})+\LL,\hx_n'=(x'_{n,G},x'_{n,H})+\LL\in\hX$ such that
$x_n,x_n'\in X$,
$\hx_n\to x+\LL$, $\hx_n'\to x'+\LL$ and $\nuW(\hx_n)\to\nu$, $\nuW(\hx_n')\to\nu'$, see Lemma~\ref{lemma:projection-basic2}.
Let $y_n=(x_G,x_{n,H})$ and $y_n'=(x_G,x'_{n,H})$  (recall that $x_G=x_G'$). 
 As
$\nuW(y_n+\LL)=\nuW(\hT_{x_G-x_{n,G}}\hx_n)=S_{x_G-x_{n,G}}\nuW(\hx_n)$,
also $\nuW(y_n+\LL)\to\nu$ and similarly $\nuW(y_n'+\LL)\to\nu'$.
As this is vague convergence, this means 
\begin{equation*}
\forall j\in\N\ \exists n_j\in\N\ \forall \ell\in\LL_j\ \forall n\geqslant n_j:\
\begin{cases}
x_{n,H}+\ell_H=(y_n+\ell)_H\in W\text{ if and only if }\nu\{x+\ell\}=1\\
\hspace{3cm}\text{and}\\
x_{n,H}'+\ell_H=(y'_n+\ell)_H\in W\text{ if and only if }\nu'\{x'+\ell\}=1\ .
\end{cases}
\end{equation*}
As $x_G=x_G'$ and as $\piG|_\LL$ is \oneone, we have 
\begin{equation*}
\nu\{x+\ell\}=1
\quad\Leftrightarrow\quad
\piG_*\nu\{x_G+\ell_G\}=1
\quad\Leftrightarrow\quad
\piG_*\nu'\{x'_G+\ell_G\}=1
\quad\Leftrightarrow\quad
\nu'\{x'+\ell\}=1\ .
\end{equation*}
Hence 
\begin{equation}
\forall \ell\in\LL\ \exists n_\ell\in\N\ \forall n\geqslant n_\ell: 1_{W}(x_{n,H}+\ell_H)=\nu\{x+\ell\}=\nu'\{x'+\ell\}=1_{W}(x_{n,H}'+\ell_H)\ ,
\end{equation}
As $(W,\LL)$ is assumed to be strongly uniquely coding and as $\nu,\nu'\neq\0$, it follows that $x_H=x_H'$.

Conversely, suppose now that $\piG_*$ is \oneone. Let $h,h',h_n,h_n'$ be as in 
Definition~\ref{def:unique-coding}b.
Let $x=(0,h)$, $x_n=(0,h_n)$, $x'=(0,h')$ and $x_n'=(0,h_n')$. Then $x_n\to x$, $x_n'\to x'$ and
\begin{equation}\label{eq:coincidence}
\forall \ell\in\LL\ \exists n_\ell\in\N\ \forall n\geqslant n_\ell: 1_{G\times W}(x_n+\ell)=1_{G\times W}(x_n'+\ell)\ ,
\end{equation}
where this common value is $1$ for at least one $\tilde{\ell}\in\LL$
and all $n\geqslant n_\ell$. Passing to subsequences, if necessary, we can assume that there are $\nu,\nu'\in\MW\setminus\{\0\}$
such that $\nuW(x_n+\LL)\to\nu$ and $\nuW(x_n'+\LL)\to\nu'$ vaguely. In particular, $\nu\leqslant\nuW(x+\LL)$ and $\nu'\leqslant\nuW(x'+\LL)$ by Lemma~\ref{lemma:projection-basic2}. Abusing slightly the notation $S_g\nu$, we define $S_y\nu$ by $(S_y\nu)(A)=\nu(-y+A)$ for all $y\in G\times H$. Let $K_1\subseteq K_2\subseteq\dots\subseteq G\times H$ be as in the first part of the proof. Then
\begin{equation*}
S_{y}\nuW(x+\LL)
=
\sum_{\ell\in\LL}1_{G\times W}(x+\ell)\cdot S_{y}\delta_{x+\ell}
=
\sum_{\ell\in\LL}1_{G\times W}(x+\ell)\cdot \delta_{y+x+\ell}\ ,
\end{equation*}
so that in view of (\ref{eq:coincidence}) there are $n(j)\in\N$ such that for each $K_j$ and all  $n\geqslant n(j)$
\begin{equation*}
\left(S_{-x_n}\nuW(x_n+\LL)\right)|_{K_j}
=
\left(S_{-x_n'}\nuW(x_n'+\LL)\right)|_{K_j}\ .
\end{equation*}
In the limit $n\to\infty$ this yields $S_{-x}\nu=S_{-x'}\nu'$. As $\piG(x)=\piG(x')=0$, this implies $\piG_*\nu=\piG_*(S_{-x}\nu)=\piG_*(S_{-x'}\nu')=\piG_*\nu'$, and as $\piG_*$ is \oneone, it follows that $\nu=\nu'\in\MW\setminus\{\0\}$. Hence $\nuW((0,h)+\LL)=\nuW((0,h')+\LL)\neq\0$, and $h=h'$ follows in the same way as at the end of the 
proof of Lemma~\ref{lemma:uniquely-coding}.
\qed

\section{Configurations with maximal density and their orbit closures}
\label{sec:Moody-proof}

\paragraph{Proof of Theorem~\ref{theo:Moody-extension}}(following the arguments in 
\cite[Proof of Theorem 1]{Moody2002})
\quad\\
If $P$ is the one-point mass concentrated on $\0\in\cM$, then the limit $D_P$ is clearly zero. Otherwise $P$ is an ergodic $S$-invariant probability measure on $\MW\setminus\{\0\}$, so that $P\circ \hat\pi^{-1}$ is a $\hT$-invariant probability measure on $\hX$, 
where $\hat\pi:\MW\setminus\{\0\}\to\hX\setminus \ZW$ is defined in Definition~\ref{def:hat-pi}.
It follows that $P\circ\hat\pi^{-1}=m_\hX$. 
 Hence it suffices to prove (\ref{eq:nu-density}) for $m_\hX$-a.e. $\hx$ and all $\nu\leqslant\nuW(\hx)$, i.e., all $\nu\in\hat\pi^{-1}\{\hx\}$. 

Observe first that
 \begin{equation}\label{eq:Moody-proof-1}
\limsup_{n\to\infty}\frac{\nu(A_n\times H)}{m_G(A_n)}
\leqslant
\limsup_{n\to\infty}\frac{\nu_{\scriptscriptstyle W}(\hx)(A_n\times H)}{m_G(A_n)}\leqslant\dL\cdot m_H(W)
\quad\text{for  all $\hx\in\hX$ and all }\nu\in\hat\pi^{-1}\{\hx\}
 \end{equation}
 by Theorem~\ref{theo:Moody}. 
Next we investigate the convergence of the l.h.s. of (\ref{eq:Moody-proof-1}) and identify its limit $D_P$\footnote{The following argument also applies to non-abelian but unimodular $G$.}.
 To this end fix a compact subset $H_0\subseteq H$ such that $\nu(G\times(H\setminus  H_0))=0$ for all $\nu\in \MW$. 
 Then fix a compact zero neighbourhood $B\subseteq G$ such that 
  $(\ell_G+B)\cap(\ell_G'+B)=\emptyset$ for all $\ell\neq\ell'$ in $\LL\cap (G\times H_0)$. 
 With $\hx=x+\LL$,  this assumption has the following consequence:
\begin{equation*}
\begin{split}
\nu\left((-\partial^BA_n)\times H_0\right)
&\leqslant
\nuW(\hx)\left((-\partial^BA_n)\times H_0\right)
= 
\sum_{y\in(x+\LL)\cap(G\times W)}\!\!\!\!\!1_{-\partial^BA_n}(y_G)
\leqslant 
\sum_{\ell\in\LL\cap(G\times (-x_H+H_0))}\!\!\!\!\!1_{-\partial^BA_n}(x_G+\ell_G)
\\
&=
\frac{1}{m_G(B)}\sum_{\ell\in\LL\cap((-x_G-\partial^BA_n)\times (-x_H+H_0))}\int 1_{x_G+\ell_G+B}(g)\,dm_G(g)\\
&\leqslant
\frac{1}{m_G(B)}\,m_G\left(\bigcup_{\ell\in\LL\cap((-x_G-\partial^BA_n)\times (-x_H+H_0))}(x_G+\ell_G+B)\right)
\leqslant
\frac{1}{m_G(B)}\,m_G\left(-\partial^BA_n+B\right)\\
&=\frac{1}{m_G(B)}\,m_G\left(-B+\partial^BA_n\right)
\leqslant
\frac{1}{m_G(B)}\,m_G\left(\partial^{-B+B}A_n\right),
\end{split}
\end{equation*}
so that
\begin{equation*}
\begin{split}
&\hspace*{-2cm}\left|
\int_{G\times H}\left(1_{-A_n}\ast 1_{B}\right)(x_G)\,d\nu(x)-m_G(B)\,\nu((-A_n)\times H)\right|\\
&\leqslant
m_G(B)\int_{G\times H}\left|\left(1_{-A_n}*\frac{1_B}{m_G(B)}\right)(x_G)-1_{-A_n}(x_G)\right|d\nu(x)\\
&\leqslant
m_G(B)\,\nu\left((-\partial^BA_n)\times H_0\right)
\leqslant
m_G\left(\partial^{-B+B}A_n\right).
\end{split}
\end{equation*}  
As $(A_n)_n$ is a van Hove sequence, this implies
\begin{equation}\label{eq:vanHove-consequence}
\begin{split}
\lim_{n\to\infty}m_G(B)\frac{\nu((-A_n)\times H)}{m_G(A_n)}
&=
\lim_{n\to\infty}\frac{1}{m_G(A_n)}\int_{G\times H}(1_{-A_n}\ast 1_B)(x_G)\,d\nu(x)\\
&=
\lim_{n\to\infty}\frac{1}{m_G(A_n)}\int_{G\times H}\left(\int_{-A_n} 1_{B}(-g+x_G)\,dm_G(g)\right)d\nu(x)\\
&=
\lim_{n\to\infty}\frac{1}{m_G(A_n)}\int_{G\times H}\left(\int_{A_n} 1_{B\times H}((g,0)+x)\,dm_G(g)\right)d\nu(x)\\
&=
\lim_{n\to\infty}\frac{1}{m_G(A_n)}\int_{A_n}\left(\int_{G\times H} 1_{B\times H}\circ T_{g}\,d\nu\right)dm_G(g)\\
&=
\lim_{n\to\infty}\frac{1}{m_G(A_n)}\int_{A_n}(S_{g}\nu)(B\times H)\,dm_G(g)
\end{split}
\end{equation}
whenever $\hx\in\hX$, $\nu\in\hat{\pi}^{-1}\{\hx\}$ and any of these limits exists. For later reference, we note that the same holds for limits along subsequences.

Observe next that the evaluation
$\nu\mapsto \nu(B\times H)=\nu(B\times H_0)$  is upper semicontinuous, in particular measurable so that, as in \cite{Moody2002},
the generalised Birkhoff ergodic theorem \cite{Lindenstrauss2001} applies to the last line of (\ref{eq:vanHove-consequence}), i.e., we have for $P$-a.e. $\nu$
\begin{equation}\label{eq:Moody-generalization}
\begin{split}
\int_{\MW}\nu'(B\times H)\,dP(\nu')
&=
\lim_{n\to\infty}\frac{1}{m_G(A_n)}\int_{A_n}(S_{g}\nu)(B\times H)\,dm_G(g)\\
&=
\lim_{n\to\infty}m_G(B)\frac{\nu((-A_n)\times H)}{m_G(A_n)}
=\lim_{n\to\infty}m_G(B)\frac{\nu(A_n\times H)}{m_G(A_n)}\ ,
\end{split}
\end{equation}
where we note for the last equation that $(-A_n)_n$ is a van Hove sequence as well if $G$ is abelian.

We just have shown that 
$D_P=\frac{1}{m_G(B)}\int\nu(B\times H)\,dP(\nu)$. 
It remains to prove that $D_P\leqslant \dL\cdot m_H(W)$ with equality if and only if $P=\QM =m_\hX\circ\nuW^{-1}$, see Definition~\ref{def:Q_M}. Observe first that $D_{\QM }=\dL\cdot m_H(W)$ by Moody's Theorem~\ref{theo:Moody}. Next observe that $\nu\leqslant\nuW(\hx)$ where $\hx=\hat\pi(\nu)$. Recalling that 
$P\circ\hat\pi^{-1}=m_\hX$, we see that
\begin{equation*}
\begin{split}
D_P
&=
\frac{1}{m_G(B)}\int_{\MW}\nu(B\times H)\,dP(\nu)
\leqslant
\frac{1}{m_G(B)}\int_{\MW}\nuW(\hat\pi(\nu))(B\times H)\,dP(\nu)\\
&=
\frac{1}{m_G(B)}\int_{\hX}\nuW(\hx)(B\times H)\,dm_\hX(\hx)
=
\frac{1}{m_G(B)}\int_{\MW}\nu(B\times H)\,d\QM (\nu)
=D_{\QM }\ ,
\end{split}
\end{equation*}
with equality if and only if $\nu=\nuW(\hat\pi(\nu))$ for $P$-a.e. $\nu$, i.e., if $P=\QM $.
\qed
\paragraph{Proof of Corollary~\ref{coro:Moody-extension}}\quad\\
Let $P$ be an ergodic $S$-invariant probability measure on the Borel-$\sigma$-algebra $\cB(\MWG)$ of $\MWG$. It can be transfered to an ergodic $S$-invariant probability $P_0$ on the sub-$\sigma$-algebra $(\piG_*)^{-1}(\cB(\MWG))\subseteq\cB(\MW)$. As $\MWG$ is a Polish space, $\cB(\MWG)$ and hence also $(\piG_*)^{-1}(\cB(\MWG))$ are countably generated.
Hence $P_0$ can be extended to a measure $P'$ on $\cB(\MW)$ \cite[Theorem 9.8.2]{bogachev2006}.

Let $\tilde{P}$ be a weak limit of a subsequence $P'_{n_i}$, where
\begin{equation*}
P'_n:=\frac{1}{m_G(A_n)}\int_{A_n} P'\circ S_{-g}\,dm_G(g)\ ,
\end{equation*}
and observe that for each continuous $\varphi:\MWG\to\R$ holds
\begin{equation*}
\begin{split}
\int_{\MW}\varphi\circ\piG_*\,dP'_n
&=
\frac{1}{m_G(A_n)}\int_{A_n}\left(\int_{\MWG} \varphi\circ S_g\,d(P'\circ(\piG_*)^{-1})\right)dm_G(g)\\
&=
\frac{1}{m_G(A_n)}\int_{A_n}\left(\int_{\MWG} \varphi\circ S_g\,dP\right)dm_G(g)
=\int_{\MWG}\varphi\,dP\ ,
\end{split}
\end{equation*}
as $P'\circ(\piG_*)^{-1}=P_0\circ(\piG_*)^{-1}=P$ and $P\circ S_{-g}=P$.

Then  $\tilde{P}$ is $S$-invariant, and for each continuous $\varphi:\MWG\to\R$ we have
\begin{equation*}
\int_{\MW}\varphi\circ\piG_*\,d\tilde{P}
=
\lim_{i\to\infty}\int_{\MW}\varphi\circ\piG_*\,d\bar{P}_{n_i}
=\int_{\MWG}\varphi\,dP\ ,
\end{equation*}
i.e., $P=\tilde{P}\circ(\piG_*)^{-1}$. If $\tilde{P}$ is not ergodic, we can decompose it into its ergodic components, $\tilde{P}=\int_{\MW}\tilde{P}_\nu\,d\tilde{P}(\nu)$.
Then $P=\int_{\MW}\tilde{P}_\nu\circ(\piG_*)^{-1}\,d\tilde{P}(\nu)$, and all $\tilde{P}_\nu\circ(\piG_*)^{-1}$ are again ergodic. As $P$ itself is ergodic, it follows that $P=\tilde{P}_\nu\circ(\piG_*)^{-1}$ for $\tilde{P}$-a.e. $\nu$.
\qed

\paragraph{Proof of Theorem~\ref{theo:Moody-topogical}}\quad\\
a)
Observe first that $m_\hX(\hXmax)=1$ by Theorem~\ref{theo:Moody-extension}.
The $\hT$-invariance of $\hXmax$ follows from
\begin{equation*}
\nuW(\hT_{-g}\hx)(A_n\times H)
=
S_{-g}\nuW(\hx)(A_n\times H)
=
\nuW(\hx)(T_g(A_n\times H))
=
\nuW(\hx)((g+A_n)\times H)
\end{equation*}
and the fact that $(g+A_n)_n$ is a van Hove sequence.\\
b) Let $\hx\in\hXmax=\{\hx\in\hX: \lim_{n\to\infty}\nuW(\hx)(A_n\times H)/m_G(A_n)=D_{\QM}\}$ and
suppose that $Q_{n_i,\hx}$ converges weakly to some probability measue $P$ on $\MW$. Fix a compact zero neighbourhood $B\subseteq G$  as in the proof of Theorem~\ref{theo:Moody-extension} and recall that $\hx\in\hXmax$.
 Then equation (\ref{eq:vanHove-consequence}) implies
\begin{equation*}
\begin{split}
m_G(B)\cdot D_{\QM}&=
\lim_{i\to\infty}m_G(B)\frac{\nuW(\hx)((-A_{n_i})\times H)}{m_G(A_{n_i})}
=
\lim_{i\to\infty}\frac{1}{m_G(A_{n_i})}\int_{A_{n_i}}(S_{g}\nuW(\hx))(B\times H)\,dm_G(g)\\
&=
\lim_{i\to\infty}\frac{1}{m_G(A_{n_i})}\int_{A_{n_i}}\int_{\MW}\nu(B\times H)\,d\delta_{S_{g}\nuW(\hx)}(\nu)\,dm_G(g)
=
\lim_{i\to\infty}\int_{\MW}\nu(B\times H)\,dQ_{n_i,\hx}(\nu)\\
&\leqslant
\int_{\MW}\nu(B\times H)\,dP(\nu)=m_G(B)\cdot D_P\ ,
\end{split}
\end{equation*}
where the inequality holds due to the Portemanteau theorem, as $\{\nu|_{B\times H}: \nu\in\MW\}$ is closed in $\MW$.
Hence 
\begin{equation*}
D_{\QM}
\leqslant
D_P
\leqslant 
\dL\cdot m_H(W)
=
D_{\QM}\ ,
\end{equation*}
This shows that $D_P=D_{\QM}$ and hence $P=\QM$ in view of Theorem~\ref{theo:Moody-extension}.\\
c)
As $Q_{n,\hx}(\MW(\hx))=1$ for all $n\in\N$ and as $Q_{n,\hx}$ converges weakly to $\QM$, we have $1\geqslant Q_M(\MW(\hx))\geqslant\lim_{n\to\infty}Q_{n,\hx}(\MW(\hx))=1$ for each $\hx\in\hXmax$.\\
d) For $\hx\in\nuW^{-1}(\supp(\QM))$ also $\MW(\hx)\subseteq\supp(\QM)$ by translation invariance of $m_\hX$ and closedness of $\supp(\QM)$. Hence 
$\MW(\hx)=\supp(\QM)$ for all $\hx\in\hXmax\cap\nuW^{-1}(\supp(\QM))$, and $m_\hX(\hXmax\cap\nuW^{-1}(\supp(\QM)))=1$, because $m_\hX(\hXmax)=1$ and
$m_\hX\circ\nuW^{-1}=\QM$.
\\
e) Let $\hx\in \CW$.
Suppose for a contradiction that $\nuW(\hx)\not\in\supp(\QM)$. Then there is an open neighbourhood $O\subseteq\cM$ of $\nuW(\hx)$ such that $\QM(O)=m_\hX(\nuW^{-1}(O))=0$. 
As $\hx\in \CW$, there is an open neighbourhood $U\subseteq\hX$ of $\hx$ such that $\nuW(U)\subseteq O$, which implies $m_\hX(U)=0$, a contradiction.
\qed

\paragraph{Proof of Remark~\ref{remark:support}}\quad\\
a) As $\CW$ is a dense $G_\delta$-set, we only have to show that $\nuW(\CW)\subseteq\overline{\nuW(R)}$ for each dense $G_\delta$-set $R\subseteq\hX$.
But this is obvious, because $R$ is dense and $\CW$ is the set of continuity points of $\nuW$.\\
b) 
If $F\subseteq\hX$ has full Haar measure, then $\QM(\nuW(F))=m_\hX(F)=1$, so that $\supp(\QM)\subseteq\overline{\nuW(F)}$. Hence 
$\supp(\QM)\subseteq\bigcap_{F\in\cF}\overline{\nuW(F)}$, where the intersections ranges over the family $\cF$ of all full measure subsets of $\hX$.
For the converse inclusion consider $F=\hXmax\cap\nuW^{-1}(\supp(\QM))$.
Then $F\in\cF$ by Theorem~\ref{theo:Moody-topogical}d, and $\nuW(F)\subseteq\supp(\QM)$ so that also $\overline{\nuW(F)}\subseteq\supp(\QM)$.\\
c) Suppose now that $m_H(\partial W)=0$. Then $\CW\subseteq\hX$ has full Haar measure. Hence assertion b) implies
\begin{equation*}
\overline{\nuW(\CW)}\subseteq\supp(\QM)=\bigcap_{F\in\cF}\overline{\nuW(F)}
\subseteq \overline{\nuW(\CW)}\ .
\end{equation*}
The same proofs apply to $\nuWG$ and $\QMG$.
\qed

\paragraph{Proof of Remark~\ref{remark:patterns} and Proposition~\ref{prop:exceptional-points}}\quad\\
We assume $W\ne\emptyset$ without loss of generality. For both proofs we need the following construction:
Let $U=U_G\times H$ with $U_G\subseteq G$ an open, relatively compact zero neighbourhood, small enough such that $\LL\cap ((-\overline{U_G}+\overline{U_G})\times (-W+W))=\{0\}$. Then the quotient map $G\times H\to \hX$, when restricted to $U\subseteq G\times H$, is \oneone. 
Given two finite disjoint sets $\LL_1,\LL_0\subseteq\LL$, not both empty,
define $O=O(\LL_1,\LL_0,U)\subseteq \MW$ by 
\begin{equation}\label{eq:O-def}
O=\left\{\nu\in \MW: \LL_1\subseteq \supp(\nu)-U,\ \LL_0\cap\supp(\nu)-U=\emptyset\right\}.
\end{equation}
Then $O$ is open in the vague topology. Let $\hat U:=\left\{\hx\in\hX: \exists u\in U\text{ with }\pihX(u)=\hx\right\}$ and note that if such an element $u\in U$ exists, then it is unique. This allows to define
$\xi:\hat U\to U$ by $\xi(\hx)=u\in(\pihX)^{-1}\{u\}$. Then
\begin{equation*}
\begin{split}
\nuW^{-1}(O)
=&
\big\{\hx=x+\LL\in\hX: x\in X,\ \forall\ell_1\in\LL_1\ \exists{\ell}\in\LL\ \exists{u}\in U:\ \ell_1=x+{\ell}-{u}\text{ and }\piH({x+\ell})\in W\\
&\hspace*{2.7cm}\text{and }
\forall\ell_0\in\LL_0\ \forall\ell\in\LL\ \forall u\in U:\ \ell_0=x+\ell-u\ \Rightarrow\ \piH(x+\ell)\not\in W\big\}\\
=&
\big\{\hx=x+\LL\in\hat U: \ \forall\ell_1\in\LL_1: \piH({\xi(\hx)+\ell_1})\in W
\text{ and }
\forall\ell_0\in\LL_0:  \piH(\xi(\hx)+\ell_0)\not\in W\big\}\\
=&
\left\{\hx\in\hat U:\; \forall\ell_1\in\LL_1: \piH(\xi(\hx)+{\ell_1})\in W\text{ and }\forall\ell_0\in\LL_0: \piH(\xi(\hx)+{\ell_0})\not\in W\right\}\\
=&
\left\{\hx\in\hat U: (\xi(\hx))_H\in \bigcap_{\ell_1\in\LL_1}(W-\ell_{1,H})\cap\bigcap_{\ell_0\in\LL_0}(H\setminus(W-\ell_{0,H})) \right\}\ .
\end{split}
\end{equation*}

As the restriction of the quotient map to $U$ is \oneone, we have by the extended Weil formula \cite[Thm.~3.4.6]{RS00}
\begin{equation}\label{eq:revised-1}
\begin{split}
\QM(O)
&=
m_\hX(\nuW^{-1}(O))
=
\dL\cdot m_{G\times H}\left\{u\in U:\ u_H\in \bigcap_{\ell_1\in\LL_1}(W-\ell_{1,H})\cap\bigcap_{\ell_0\in\LL_0}(H\setminus(W-\ell_{0,H}))\right\}\\
&=
\dL\cdot m_G(U_G)\cdot m_H\left(\bigcap_{\ell_1\in\LL_1}(W-\ell_{1,H})\cap\bigcap_{\ell_0\in\LL_0}(H\setminus(W-\ell_{0,H}))\right)\ .
\end{split}
\end{equation}
\begin{proof}[Proof of Remark~\ref{remark:patterns}]
Let $\LL_1\subseteq\LL$ be nonempty finite and $\LL_0=\emptyset$.
We can choose $U_G$ such that $m_{G}(\partial U_G)=0$, compare \cite[Lemma~3.6]{HuckRichard15}. Then $O=O(\LL_1,\emptyset,U)$ is a continuity set for $\QM$, as is seen from the above calculation with $\partial O=\{\nu\in\MW: \LL_1\subseteq \supp(\nu)-\partial U\}$.
Next, note that we have asymptotically as $n\to\infty$
\begin{displaymath}
Q_{n,\hx}(O)=f_n(\LL_1, \nuW(\hx))\cdot m_G(U_G)+o(1)
\end{displaymath}
on any van Hove sequence.
Hence weak convergence in Theorem~\ref{theo:Moody-topogical}b  yields for any $\hx\in \hXmax$ that $\lim_{n\to\infty}Q_{n,\hx}(O)=\QM(O)$ on the associated tempered van Hove sequence. We thus obtain
\begin{equation*}
\begin{split}
f(\LL_1,\nuW(\hx))
&=
\frac{\QM(O)}{m_G(U_G)}
=
\dL\cdot m_H\left(\bigcap_{\ell_1\in\LL_1}(W-\ell_{1,H})\right)\ .
\end{split}
\end{equation*}
\end{proof}

\begin{proof}[Proof of Proposition~\ref{prop:exceptional-points}]\quad\\
For any given point $\hx=x+\LL\in\hX$ and any nonempty finite set $\LL'\subseteq \LL$,  let $\LL_1=\{\ell\in\LL': \ell_H \in -x_H+W\}$ and 
$\LL_0=\LL'\setminus\LL_1$.
Then, for any $h\in H$,
\begin{equation*}
(-h+W)\cap\piH(\LL')=\piH(\LL_1)
\quad\text{if and only if}\quad
h\in \bigcap_{\ell \in\LL_1}(W-\ell_{H})\cap
\bigcap_{\ell\in\LL_0}\left(H\setminus(W-\ell_H)\right)\ .
\end{equation*}
As $\piH(\LL_1)=(-x_H+W)\cap \piH(\LL')$, we thus have
\begin{equation}\label{eq:revised-3}
\left\{h\in H: (-x_H+W)\cap\piH(\LL')=(-h+W)\cap\piH(\LL')\right\}=\bigcap_{\ell \in\LL_1}(W-\ell_{H})\cap
\bigcap_{\ell\in\LL_0}\left(H\setminus(W-\ell_H)\right)\ .
\end{equation}
``$\Longrightarrow$''
Suppose that $\nuW(\hx)\in\supp(\QM)$.
Define $O=O(\LL_1,\LL_0,U)$ as in (\ref{eq:O-def}).
As $O\subseteq\MW$  is open and as $\nuW(\hx)\in O$ by definition of this set, we have $\QM(O)>0$, which implies
in view of (\ref{eq:revised-1}) and (\ref{eq:revised-3}) that
\begin{equation*}
m_H\left\{h\in H: (-x_H+W)\cap\piH(\LL')=(-h+W)\cap\piH(\LL')\right\}>0\ .
\end{equation*}
\noindent ``$\Longleftarrow$''  
Suppose that $\nuW(\hx)\not\in\supp(\QM)$. Then
there is an open neighbourhood $O_0\subseteq\cM$ of $\nuW(\hx)$ such that $\QM(O_0)=m_\hX(\nuW^{-1}(O_0))=0$.
There are finite disjoint sets $\LL_1,\LL_0\subseteq\LL$, not both empty, such that the open set $O=O(\LL_1,\LL_0,U)$ is contained in $O_0$, with $U=U_{G}'\times H$ and $U_G'\subseteq U_G$ a sufficiently small open zero neighborhood, compare (\ref{eq:O-def}). Hence, in view of  (\ref{eq:revised-1}) and (\ref{eq:revised-3}),
\begin{equation*}
m_H\left\{h\in H: (-x_H+W)\cap\piH(\LL')=(-h+W)\cap\piH(\LL')\right\}
=0\ .
\end{equation*}
\end{proof}

\section{Outlook}\label{sec:outlook}
\paragraph{The topological theory}
Theorem~\ref{theo:projection-top} shows that the notion of a maximal equicontinuous factor is of no help for studying the system $(\MW,S)$, when $\inn(W)$, the interior of the window, is empty, because then this factor is trivial. This holds, a fortiori, for the system $(\MWG,S)$ which is a factor of $(\MW,S)$.
On the other hand, in view of Proposition~\ref{prop:projection-homeo}, the system $(\GMW,S)$ has always the maximal equicontinuous factor $(\hX,\hT)$,
and in view of Lemma~\ref{lemma:proj-homeo-1} the same is true for the system $(\GMWG,S)$.

This means that, if we 
want to study more general transitive subsystems $(M,S)$ of $(\cMG,S)$, where
$\cMG$ denotes the space of all locally finite measures on $G$ endowed with the topology of vague convergence, one should look for a way to extend $(M,S)$ (which generalises $(\MWG,S)$) in a natural way to a system that plays the role of $(\GMWG,S)$ (and hence has a maximal equicontinuous factor). A hint how to accomplish this is given
in the proof of Theorem~\ref{theo:projection-top}b. It is shown there that $\GMWG\subseteq\hX\times \MWG$ is a nontrivial joining between these two systems, except if $\card(\MWG)=1$.

So one might attempt to look for
maximal equicontinuous systems that can be joined in a nontrivial and non-redundant way to a general transitive system $(M,S)$. By \emph{non-redundant} we mean that the joining is such that for a dense $G_\delta$-set of points in $\mu\in M$ the $\mu$-section of the joining consists of just one point and that the joining itself is the closure of the set of all these points. Observe that the joining $\GMW\subseteq\hX\times \MW$ does have this property, because the zero-measure, if it belongs to $\MW$,  is the only element in $\MW$ that is joined to more than one point in $\hX$ and $\GMW\setminus(\hX\times\{\0\})$
is dense in $\GMW$.

An equivalent but more appealing way to approach this problem would be to investigate existence and properties of a \emph{generic maximal equicontinuous factor}
of $(M,S)$ (see \cite{HuangYe2012} for the notion of generic factors).
Systems $(M,S)$ generated by relatively compact but non-compact windows might be first candidates to test such ideas.

\paragraph{Ergodic theory} 
If $P$ is an $S$-invariant probability measure on $\GMW$, then the system $(\GMW,P,S)$ has the structure of a fibred random dynamical system 
over the base $(\hX,m_\hX,\hT)$ with fibres which are subsets of copies of 
$\{0,1\}^\LL$. This sets the scene to study entropy properties of such systems along the lines developed e.g.~in \cite{Ward1992},
and \cite{Bogenschuetz1992} can be the starting point for a thermodynamic formalism - at least when $G=\Z$.
Recent results on $\mathscr B$-free systems such as \cite{Kulaga-Przymus2014,BKKL2015} might provide a good testing ground for such ideas.

\paragraph{Non-abelian model sets}

Our analysis carries over to the non-abelian case, see \cite[Ch.~7]{Wolf07} for the basic setting.  Concerning topological results, note that the compact minimal left coset dynamical system $(\hX,\hT)$ might no longer be equicontinuous, but remains distal. We can then still consider the graph dynamical system $(\GMW, S)$, and the results of Sections~\ref{sec:in}, \ref{subsec:top-reg-wind} and Section~\ref{sec:top} remain true, the latter with ``equicontinuous'' replaced by ``distal''.  For measure theoretic results, note that the left invariant quotient measure $m_{\hX}$ on $\hX$ remains unique up to normalisation \cite[Sec.~3.4]{Wolf07}. Hence $\hX$ together with its natural left action is strictly ergodic.  (This is true without cocompactness of the discrete group $\LL$.) Thus Theorem~\ref{theo:projection-mth} still holds -- up to the spectral statement in part a).
Also Theorems \ref{theo:Moody}, \ref{theo:Moody-extension}, \ref{theo:Moody-topogical} remain true with minor adjustments, compare the proof of Theorem~\ref{theo:Moody-extension}. This implies that the entropy estimate for weak model sets \cite{HuckRichard15}, together with its proof,  continues to hold in that case.  Note however that the latter results rely on the existence of a van Hove sequence, see \cite{mr13} for a discussion of this strong form of amenability in the non-abelian setting. 

In the recent work \cite{BHP16}, Schlottmann's approach to the torus parametrisation \cite{Schlottmann00} is revisited for non-abelian regular model sets, with a  measure theoretic focus.  Then diffraction theory is developed in a very general setting, which does not assume amenability  nor cocompactness. In the cocompact case, a pure point diffraction result is obtained for groups $G$ admitting a Gelfand pair. It might be interesting to further explore the relation between suitable types of dynamical and diffraction spectrum for model sets, thereby extending the pure point result of the abelian case as in \cite{BaakeLenz2004}. 
A natural starting point for such an analysis seems diffraction of
lattices, compare the  abelian result \cite{RS15}. Also the recent spectral analysis of odometer actions \cite{DL13} might provide further insight.

\end{document}